\documentclass[11pt]{article}
\usepackage{amsmath}
\usepackage{amssymb}
\usepackage{bm}
\usepackage{cite}
\usepackage{verbatim}
\usepackage{algorithm,algpseudocode}
\usepackage{amsthm}
\usepackage{graphicx}
\usepackage{subfigure}
\usepackage{float}
\usepackage{tikz}
\usepackage{cite}
\usepackage{xcolor}
\usepackage{verbatim}
\usepackage{mathrsfs}
\usepackage{hyperref}
\usepackage{placeins}
\usepackage[nohead,margin=1.0in]{geometry}
\usepackage[title]{appendix}

\newtheorem{theorem}{Theorem}[section]

\newtheorem{lemma}[theorem]{Lemma}
\newtheorem*{remark}{Remark}
\newcommand\keywords[1]{\textbf{Keywords}: #1}

\title{Generative Prior-Guided Neural Interface Reconstruction for 3D Electrical Impedance Tomography
}

\author{
Haibo Liu\footnotemark[2] \textsuperscript{,}\footnotemark[1]
\and
Junqing Chen\footnotemark[2]
\and
Guang Lin\footnotemark[3] \textsuperscript{,}\footnotemark[1]
}

\begin{document}
    \maketitle

    \footnotetext[1]{Corresponding authors.}
    
    \footnotetext[2]{Department of Mathematical Sciences, Tsinghua University,
    Beijing 100084, P.R. China. (\url{liuhb19@mails.tsinghua.edu.cn}, \url{jqchen@tsinghua.edu.cn}).}
    
    \footnotetext[3]{Department of Mathematics, School of Mechanical Engineering, Purdue University,
    610 Purdue Mall, West Lafayette, IN 47907, USA. (\url{guanglin@purdue.edu})}

    \begin{abstract}
    Reconstructing complex 3D interfaces from indirect measurements remains a grand challenge in scientific computing, particularly for ill-posed inverse problems like Electrical Impedance Tomography (EIT). Traditional shape optimization struggles with topological changes and regularization tuning, while emerging deep learning approaches often compromise physical fidelity or require prohibitive amounts of paired training data. We present a transformative ``solver-in-the-loop'' framework that bridges this divide by coupling a pre-trained 3D generative prior with a rigorous boundary integral equation (BIE) solver. Unlike Physics-Informed Neural Networks (PINNs) that treat physics as soft constraints, our architecture enforces the governing elliptic PDE as a hard constraint at every optimization step, ensuring strict physical consistency. Simultaneously, we navigate a compact latent manifold of plausible geometries learned by a differentiable neural shape representation, effectively regularizing the ill-posed problem through data-driven priors rather than heuristic smoothing. By propagating adjoint shape derivatives directly through the neural decoder, we achieve fast, stable convergence with dramatically reduced degrees of freedom. Extensive experiments on 3D high-contrast EIT demonstrate that this principled hybrid approach yields superior geometric accuracy and data efficiency which is difficult to achieve using traditional methods, establishing a robust new paradigm for physics-constrained geometric discovery.
    \end{abstract}
    
    
    \section{Introduction}
    
    PDE-constrained shape optimization problems represent a fundamental class of inverse problems with broad applications across scientific computing, medical imaging, and engineering design. These problems seek to recover geometric interfaces or boundaries from indirect measurements, where the relationship between geometry and observations is governed by partial differential equations. The mathematical challenge lies in the severe ill-posedness of these inverse problems, compounded by the high-dimensional nature of shape spaces and the computational complexity of repeatedly solving forward PDEs during optimization. Electrical impedance tomography (EIT) exemplifies this challenge, with transformative applications in noninvasive medical diagnosis \cite{Frerichs2000, Gao2015} and non-destructive industrial testing \cite{dickin1996electrical, karhunen2010electrical}. The technique aims to recover conductivity distributions from boundary measurements, with the high-contrast interface reconstruction problem being especially critical—where conductivity exhibits sharp discontinuities across material boundaries characterized by piecewise constant distributions $\sigma(x) = \sigma_1 \chi_{D \setminus S}(x) + \sigma_2 \chi_{S}(x)$ with open subset $S \subset D$ and smooth boundary $\partial S$. This formulation is particularly relevant for applications including metallic implant monitoring in biological tissues, tumor localization with distinct electrical properties, and industrial defect detection, where accurate geometric boundary determination is paramount for diagnostic and safety outcomes.

    A range of approaches have been developed to address the EIT interface reconstruction problem, with regularized Newton methods enhancing stability through Tikhonov regularization \cite{eppler2005regularized} and extended edge-preserving total variation constraints \cite{hintermuller2015robust}, while integral equation approaches \cite{eckel2007nonlinear} offer computational advantages by avoiding full-domain discretization and establishing Fréchet differentiability of the measurement operator. Shape optimization methods \cite{eppler2007shape} directly evolve interface geometry using gradient flows derived from shape calculus, with the level-set method \cite{osher2004level} representing a particularly powerful framework that gained prominence in EIT through pioneering work \cite{soleimani2006narrow, rahmati2012level, liu2017parametric}. Level-set-based shape optimization \cite{chen2009level, rymarczyk2018solving} combines topological flexibility with rigorous mathematical frameworks, treating the interface $\Gamma = \partial S$ as a manifold with shape derivatives derived through adjoint analysis, providing theoretical convergence guarantees through Fréchet differentiable shape derivatives while elegantly handling topological changes. Computational advances include second-order curvature acceleration techniques \cite{afraites2008second}, edge-preserving total variation regularization that significantly outperforms traditional Tikhonov methods in preserving sharp material interfaces, and recent extensions demonstrating versatility through multi-physics formulations \cite{liu2015multi}, nonstationary approaches for resolving depth ambiguities \cite{liu2019nonstationary}, and incorporation of structural priors for complex industrial applications \cite{kolehmainen2019incorporating, alghamdi2024spatial}. However, critical limitations impede their application to comprehensive 3D problems: explicit discretization imposes prohibitive memory requirements largely confining existing EIT studies to 2D or simplified quasi-3D geometries, current heuristic regularization techniques inadequately preserve complex anatomical features while frequently introducing geometric artifacts that compromise diagnostic precision, and these methods exhibit marked sensitivity to initial conditions and topological constraints during interface evolution, creating formidable convergence challenges.
        
    Two dominant deep learning paradigms have emerged for addressing inverse problems, especially electrical impedance tomography (EIT) reconstruction problems. End-to-end learning approaches \cite{hamilton2018deep, molinaro2023neural} directly map boundary voltage measurements to conductivity distributions through deep neural networks, leveraging convolutional architectures, U-Net structures, and autoencoder frameworks to achieve remarkable computational efficiency compared to traditional iterative methods. However, these methods typically require vast quantities of paired training data (ground truth conductivity distributions and corresponding boundary measurements) and exhibit limited interpretability as "black boxes." Physics-Informed Neural Networks (PINNs) \cite{lu2021physics, ZHENG2024112751} represent the alternative paradigm, parameterizing unknown conductivity fields with neural networks while incorporating the governing elliptic PDE as soft constraints through residual minimization in the loss function. While PINN extensions for EIT have shown promising results in handling complex geometries and multi-frequency measurements, training remains sensitive to the complex interplay between different loss terms (measurement misfit versus PDE residual) and hyperparameter tuning, often leading to convergence difficulties or physically inconsistent conductivity reconstructions. A complete introduction can be found in \cite{denker2025deep}.
    
    Recent advances in generative modeling have opened new avenues for solving inverse problems by learning data-driven priors from datasets. Generative approaches learn statistical regularities of feasible solutions and constrain optimization to learned manifolds of realistic configurations. However, existing generative approaches often struggle with incorporating complex physics constraints while maintaining computational efficiency. Our framework represents a principled hybrid architecture that occupies a distinct and highly advantageous middle ground between these two dominant paradigms by leveraging 3D generative models. Rather than learning explicit generative models of conductivity distributions, we employ neural implicit functions as learned geometric priors that inherently encode statistical regularities of anatomical shapes while maintaining differentiability for physics-based optimization. Neural implicit functions, which represent geometry as the level set of a continuous and differentiable mapping learned by a neural network, have shown exceptional capability in capturing intricate 3D shapes \cite{chen2019learning}. Unlike end-to-end methods, our approach is far more data-efficient, leveraging a pre-trained shape model but performing the actual inversion through a physics-based optimization loop, obviating the need for extensive paired measurement-shape data—a critical advantage in medical imaging where such data is scarce or impossible to acquire. In contrast to PINNs, our framework elegantly avoids the convergence pitfalls by maintaining the elliptic PDE governing the electric potential as a hard constraint, rigorously enforced at each optimization step by a dedicated, high-fidelity BIE solver, while the neural network's role is judiciously confined to representing the geometric interface. This modular architecture represents more than a mere design choice; it is a philosophical advantage for scientific machine learning, enabling independent improvement and substitution of each component while ensuring every proposed solution is physically consistent. By parameterizing the conductivity interface as $\Gamma = \{\bm{x} \in \mathbb{R}^3 | f_\theta(\bm{x}) = 0\}$ using a differentiable neural network $f_\theta: \mathbb{R}^3 \to \mathbb{R}$, our framework achieves a fundamental paradigm shift by being inherently three-dimensional and completely mesh-free—eliminating the dimensional constraints that have historically limited EIT reconstruction to simplified geometries. This integration of interface inverse problems with neural representation of surfaces significantly reduces the dimensionality of optimization variables through latent representations, making the EIT problem more tractable with faster convergence and improved stability, while providing explicit gradient expressions that enable efficient gradient descent algorithms for navigating the latent space to identify optimal interface configurations.

    Compared with classical level-set methods and Newton schemes with shape Hessians that optimize directly in mesh/level-set spaces with hand-tuned regularizers, our latent-space shape calculus brings some concrete advantages tailored to 3D EIT: It reduces the optimization from infinite geometric degrees of freedom to a compact latent vector, yielding substantial memory and time savings per BIE-based forward/adjoint solve; The 3D generative model produces smooth, well-defined surface normals almost everywhere, suppressing discretization artifacts that destabilize shape derivatives; The generative model’s diffeomorphic flow preserves topology during updates, avoiding the spurious splits/merges often observed in unconstrained level sets; We push adjoint shape-derivative integrands through the decoder to obtain latent-space gradients aligned with the learned manifold, improving conditioning and step efficiency. Empirically, these properties translate into marked stability and efficiency gains: we observe effective parameter reduction with rapid early descent and robust convergence under 20\% noise, and consistently low Hausdorff and volume errors in the reconstruction of complex geometries(Figures~\ref{example1:error}, \ref{noise:error}), while classical level-set pipelines typically require heavy TV/Tikhonov tuning and remain sensitive to initialization; In short, our approach keeps the PDE as a hard constraint via a matched BIE forward/adjoint pair but confines updates to a topology-preserving generative manifold, delivering robustness and data efficiency beyond traditional shape optimization.

    The key innovations and advantages of this work can be summarized as follows:
    \begin{itemize}
        \item \textbf{Principled Hybrid Architecture}: A novel "solver-in-the-loop" framework that strategically combines pre-trained neural shape priors with physics-based optimization, achieving superior data efficiency without requiring extensive paired measurement-shape training data.
        
        \item \textbf{Hard Physics Constraint Enforcement}: Maintains governing elliptic PDEs as hard constraints through dedicated BIE solvers at each optimization step, ensuring physical consistency and avoiding the convergence difficulties of soft constraint approaches.
        
        \item \textbf{Implicit Neural Interface Representation}: Parameterizes conductivity interfaces as level sets of neural implicit functions, enabling inherently three-dimensional, completely mesh-free reconstruction that eliminates dimensional constraints and discretization artifacts.
        
        \item \textbf{Efficient Latent Space Optimization}: Significantly reduces optimization dimensionality through neural latent representations, providing explicit gradients for efficient navigation of the solution space with improved convergence and stability.
    \end{itemize}
    
    The remainder of this paper is organized as follows. Section \ref{problem} formulates the EIT interface reconstruction problem as a PDE-constrained shape optimization framework. Section \ref{sec:method} presents our implicit neural paradigm, which incorporates a differentiable neural distance function for efficient shape representation. In Section \ref{sec:analysis}, we establish rigorous convergence guarantees for our optimization scheme, demonstrating that it converges to stationary points of the regularized objective function under appropriate conditions. Section \ref{sec:results} validates our method through comprehensive numerical experiments. Finally, Section \ref{sec:conclusion} discusses the broader implications for inverse problem solving and explores potential extensions to other imaging modalities.
    
    \section{The Problem Setting}\label{problem}

    In Electrical Impedance Tomography (EIT), the conductivity distribution within a domain determines the electric potential field generated by applied currents. The electric potential $u$ satisfies the following elliptic equation\cite{cheney1999electrical}:
    \begin{equation}
        -\nabla \cdot (\sigma(x) \nabla u) = 0 \quad \text{in } \Omega,
    \end{equation}
    where $\sigma(x)$ is the conductivity distribution function within the domain $\Omega$. 
    
    We consider a high-contrast conductivity model \cite{chen2009level} that captures the essential physics of EIT interface reconstruction problems. In this formulation, the conductivity exhibits sharp discontinuities across material interfaces, creating the challenging inverse problem of boundary detection from electrical measurements. 
    
    Let $D \subset \mathbb{R}^3$ be a simply connected bounded domain containing an unknown inclusion $S$ with a $C^2$-regular boundary. To ensure well-posedness of the forward problem, we impose the geometric separation condition $\text{dist}(\partial S, \partial D) > \delta > 0$ for some fixed $\delta > 0$ (see Figure \ref{fig:domain}). The high-contrast conductivity distribution is then characterized by the piecewise constant function:
    \begin{equation}
        \sigma(x) = 
        \begin{cases} 
            \sigma_+ & x \in S \\
            1 & x \in D \setminus \overline{S}
        \end{cases}, \quad \text{with } \sigma_+ \gg 1,
    \end{equation}
    where the large contrast ratio $\sigma_+ \gg 1$ represents scenarios such as metallic implants in biological tissue or highly conductive inclusions in industrial materials.
    
    The resulting conducting domain $\Omega = D \setminus \overline{S}$ exhibits a complex geometry: it is bounded externally by the accessible boundary $\Sigma := \partial D$, where electrodes can be placed for current injection and voltage measurement, and internally by the unknown interface $\Gamma := \partial S$ that we seek to reconstruct. \textcolor{black}{We define the unit normal $n$ on the boundary $\partial\Omega = \Sigma \cup \Gamma$ to be the outward-pointing normal to the conducting domain $\Omega$.} This geometric configuration encapsulates the fundamental challenge of EIT: inferring the internal structure $\Gamma$ from boundary observations on $\Sigma$, a severely ill-posed inverse problem that requires sophisticated regularization techniques for stable reconstruction.
    
    This configuration models practical scenarios where highly conductive inclusions are embedded in moderate-conductivity media, such as titanium orthopedic implants ($\sigma \approx 4 \times 10^6$ S/m) within muscle tissue ($\sigma \approx 0.5$ S/m) \cite{gabriel1996electrical}. Our framework extends naturally to general conductivity distributions.
    \begin{figure}[htb!]
        \centering
        \includegraphics[width=0.5\textwidth]{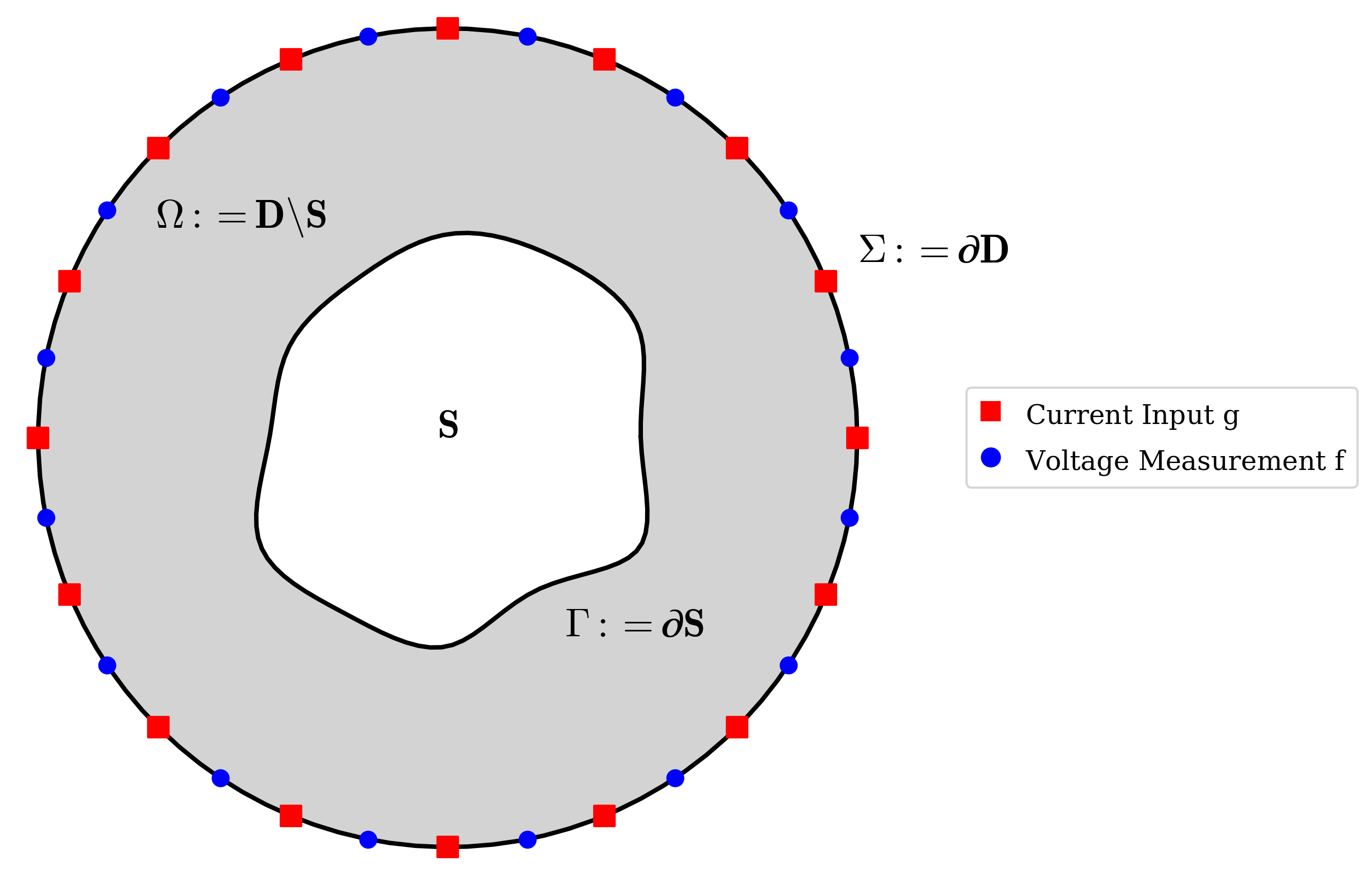}
        \caption{Geometric configuration of the 3D EIT problem (2D cross-sectional view): The conducting domain $\Omega = D \setminus \overline{S}$ (shaded) is bounded externally by $\Sigma := \partial D$ and internally by the interface $\Gamma := \partial S$. Current injection electrodes and voltage measurement points are distributed on $\Sigma$.}
        \label{fig:domain}
    \end{figure}
    
    The forward problem in EIT involves solving a series of elliptic boundary value problems in the conducting domain $\Omega$. For our analysis, we focus on the current injection scenario, which corresponds to the typical experimental setup. Given current patterns $g \in H^{-1/2}(\Sigma)$ satisfying $\int_{\Sigma} g \, ds = 0$ (current conservation), the governing equations for the electric potential $u$ are:
    \begin{equation}\label{state1}
        \begin{cases} 
            -\Delta u = 0 & \text{in } \Omega \\
            u = 0 & \text{on } \Gamma \\
            \partial_n u = g & \text{on } \Sigma
        \end{cases}.
    \end{equation}
    \textcolor{black}{As the contrast ratio approaches infinity, the electric field inside the highly conductive inclusion becomes negligible, making the potential $u$ approximately constant on the interface $\Gamma$. By fixing the gauge, this constant can be set to zero, which leads to the homogeneous Dirichlet boundary condition on $\Gamma$.} The boundary measurements of the voltage $f = u|_{\Sigma} \in H^{1/2}(\Sigma)$ are then collected for the inverse problem.
    
    The inverse problem seeks to recover the interface $\Gamma := \partial S$ from boundary measurements. The standard approach formulates this as an optimization problem that minimizes the misfit between computed and measured potentials:
    \begin{equation}\label{loss1}
        J(\Gamma) := \frac{1}{2} \| u -f \|^2_{L^{2}(\Sigma)} \to \inf_{\Gamma \in \mathcal{A}},
    \end{equation}
    where $\mathcal{A}$ denotes the admissible set of interfaces defined as:
    \begin{equation*}
        \mathcal{A} := \{\Gamma \in C^2 : \text{dist}(\Gamma, \Sigma) > \delta\}.
    \end{equation*}
    
    \subsection{Boundary Integral Formulation}\label{subsec:BIE}
    
    To efficiently solve the system in equation \eqref{state1}, we employ a boundary integral equation (BIE) approach. This powerful methodology reduces the dimensionality of the computational domain from the three-dimensional region $\Omega$ to its two-dimensional boundaries $\Gamma$ and $\Sigma$, offering significant computational advantages for interface problems.
    
    Let $G(\bm{x},\bm{y}) = (4\pi|\bm{x}-\bm{y}|)^{-1}$ denote the fundamental solution of the 3D Laplacian. For any $\phi \in H^{1/2}(\Gamma \cup \Sigma)$, we define the single-layer and double-layer potentials:
    \begin{align}
        \mathcal{S}\phi(\bm{x}) &:= \int_{\Gamma \cup \Sigma} G(\bm{x},\bm{y})\phi(\bm{y})\,\mathrm{d}s(\bm{y}), \quad \bm{x} \in \Omega \\
        \mathcal{K}\phi(\bm{x}) &:= \int_{\Gamma \cup \Sigma} \partial_{n(\bm{y})}G(\bm{x},\bm{y})\phi(\bm{y})\,\mathrm{d}s(\bm{y}), \quad \bm{x} \in \Omega 
    \end{align}
    The boundary integral operators are then defined by taking appropriate traces of these potentials:
    \begin{align}
        \mathcal{S}_{AB}\phi(\bm{x}) &:= \int_B G(\bm{x},\bm{y})\phi(\bm{y})\,\mathrm{d}s(\bm{y}), \quad \bm{x} \in A \\
        \mathcal{K}_{AB}\phi(\bm{x}) &:= \int_B \partial_{n(\bm{y})}G(\bm{x},\bm{y})\phi(\bm{y})\,\mathrm{d}s(\bm{y}), \quad \bm{x} \in A
    \end{align}
    where $(A,B) \in \{(\Gamma,\Gamma), (\Gamma,\Sigma), (\Sigma,\Gamma), (\Sigma,\Sigma)\}$, representing all possible boundary interactions.
    
    For the system governing $u$, the BIE formulation takes the elegant block form:
    \begin{equation}\label{bie_u}
        \begin{pmatrix} 
            \mathcal{S}_{\Gamma\Gamma} & -\mathcal{K}_{\Gamma\Sigma} \\ 
            -\mathcal{S}_{\Sigma\Gamma} & \frac{1}{2}\mathcal{I} + \mathcal{K}_{\Sigma\Sigma}
        \end{pmatrix}
        \begin{pmatrix}
            \partial_n u|_\Gamma \\ 
            u|_\Sigma 	
        \end{pmatrix}
        =
        \begin{pmatrix} 
            \frac{1}{2}\mathcal{I} + \mathcal{K}_{\Gamma\Gamma} & -\mathcal{S}_{\Gamma\Sigma} \\ 
            -\mathcal{K}_{\Sigma\Gamma} & \mathcal{S}_{\Sigma\Sigma}
        \end{pmatrix}
        \begin{pmatrix}
            0 \\ 
            g
        \end{pmatrix}.
    \end{equation}
    This formulation reveals the severe ill-posedness inherent to the EIT problem. \textcolor{black}{While the fundamental solution $G(\bm{x},\bm{y})$ decreases algebraically as $O(|\bm{x}-\bm{y}|^{-1})$, the off-diagonal blocks $\mathcal{S}_{\Sigma\Gamma}$ and $\mathcal{K}_{\Sigma\Gamma}$ map information between strictly disjoint boundaries $\Gamma$ and $\Sigma$. Because the integral kernels are infinitely smooth for $\bm{x} \neq \bm{y}$, these operators are highly compact \cite{kress1989linear}. Consequently, their singular values decay exponentially, which physically corresponds to the exponential attenuation of high-frequency spatial components of the interior information across the distance \cite{Borcea2002}.} As a result, small perturbations in boundary measurements can correspond to vastly different internal configurations, making the inverse problem severely ill-posed and extremely sensitive to measurement noise. \cite{Borcea2002}.
    \begin{figure}[htb!]
        \centering
        \includegraphics[width=0.5\textwidth]{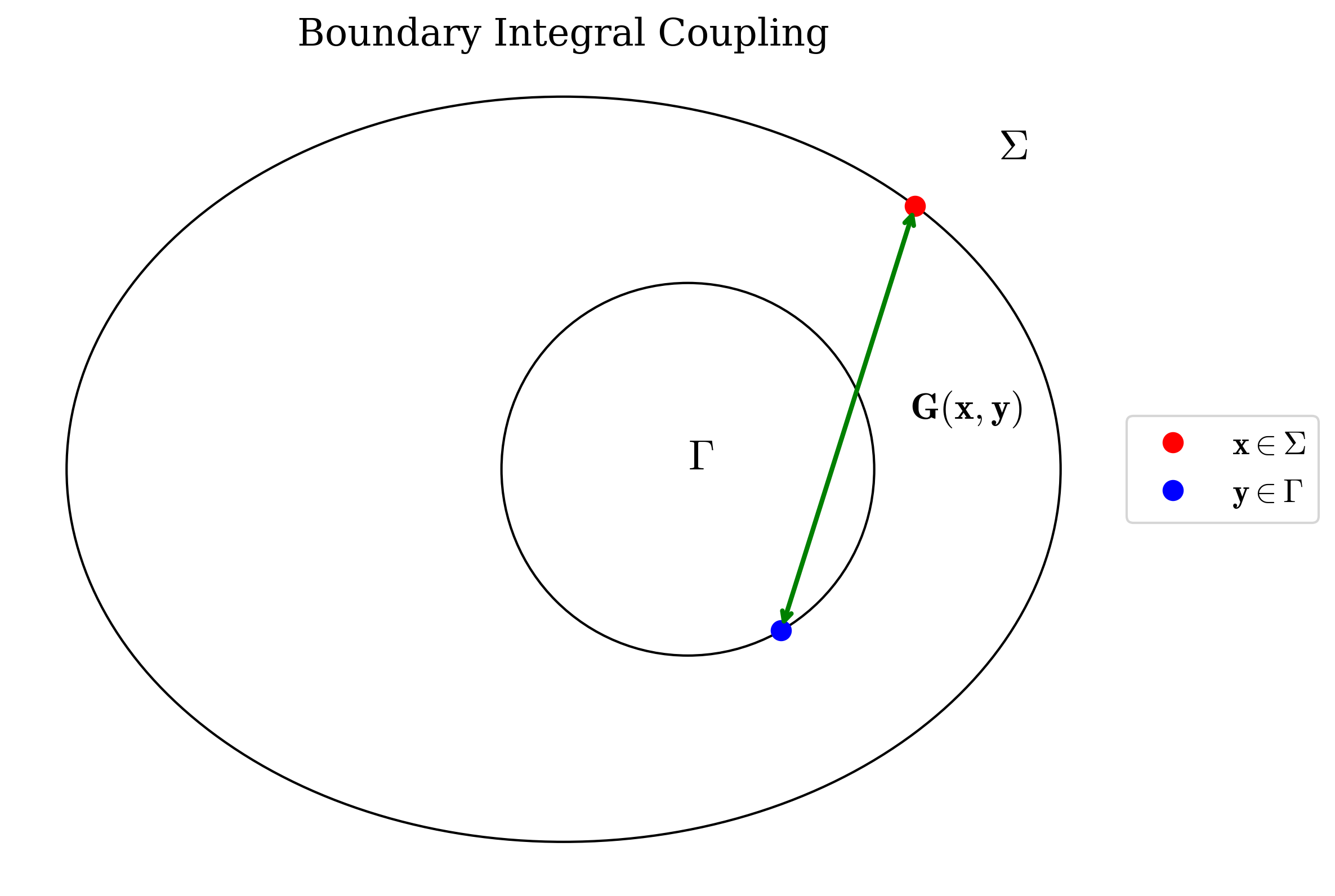}
        \caption{Illustration of boundary integral coupling: The operators $\mathcal{S}_{AB}$ and $\mathcal{K}_{AB}$ capture the interaction between boundaries $A$ and $B$ through the fundamental solution $G(\bm{x},\bm{y})$.}
        \label{fig:bie}
    \end{figure}
    
    \section{Generative Prior-Guided Neural Interface Reconstruction}\label{sec:method}
    
    The EIT inverse problem is severely ill-posed, requiring regularization to obtain unique and stable solutions. Rather than relying on hand-crafted priors (e.g., Tikhonov or Total Variation), we employ a pre-trained 3D generative model as a learned prior over plausible interfaces. This shifts the problem from unconstrained evolution in an infinite-dimensional space to a principled search on a compact, data-driven shape manifold. By integrating classical shape optimization with generative priors, we establish a rigorous framework for PDE-constrained interface reconstruction in 3D.
    
    \subsection{3D Generative Models as Learned Shape Priors}
    
    A critical component of effective EIT reconstruction is the choice of prior over interfaces. Traditional 3D approaches face notable difficulties under noise and sparse measurements. Mesh-based methods struggle with evolving complex geometries, and level-set methods can introduce artifacts. In contrast, a pre-trained 3D generative model provides a data-driven prior that concentrates the search on plausible shapes.
    
    To overcome these fundamental limitations, we employ neural implicit shape representations, specifically the Signed Distance Function (SDF) paradigm. An SDF $f: \mathbb{R}^3 \to \mathbb{R}$ elegantly maps each point in space to its signed distance from the nearest point on the interface, following the convention:
    \begin{equation}
        f(x) = 
        \begin{cases}
        -d(x, \Gamma) & \text{if } x \in S, \\
        d(x, \Gamma) & \text{if } x \in \mathbb{R}^3 \setminus S,
        \end{cases}
    \end{equation}
    where $d(x, \Gamma)$ denotes the Euclidean distance from point $x$ to the interface $\Gamma = \partial S$. This mathematically elegant representation implicitly defines the interface as the zero level set: $\Gamma = \{x \in \mathbb{R}^3 | f(x) = 0\}$.
    
    Following the pioneering DeepSDF approach \cite{park2019deepsdf}, we parameterize this function using a deep neural network $f_\theta: \mathbb{R}^3 \to \mathbb{R}$ with parameters $\theta$. This network architecture learns to approximate the signed distance function throughout the domain, enabling a remarkably compact yet highly expressive representation of complex geometrical structures. A key innovation in this framework is the incorporation of a latent code $z \in \mathbb{R}^d$ that efficiently encodes shape information within a low-dimensional manifold:
    \begin{equation}
    f_\theta: \mathbb{R}^d \times \mathbb{R}^3 \to \mathbb{R}, \quad (z, x) \mapsto f_\theta(z, x).
    \end{equation}
    This latent code $z$ becomes the optimization variable in our EIT reconstruction, allowing us to work in a compact representation space rather than directly manipulating high-dimensional geometric parameters.
    
    While DeepSDF provides a powerful representation framework, it does not inherently guarantee topological consistency during shape manipulation or evolution. This limitation is particularly problematic for EIT interface reconstruction, where preserving the correct topology of anatomical structures or industrial components is essential for accurate diagnosis or defect detection. To address this limitation, we adopt the Neural Diffeomorphic Flow (NDF) framework introduced by Sun et al. \cite{sun2022topology}, which ensures robust topology preservation through carefully designed diffeomorphic constraints. The NDF approach conceptualizes shape deformation as a differentiable flow $\psi_z: \mathbb{R}^3 \to \mathbb{R}^3$, parameterized by the latent code $z$. This flow transforms a template shape into the target shape while strictly preserving topological characteristics.
    
    The cornerstone of the NDF framework is its specialized loss function, which integrates SDF approximation with topological preservation guarantees. The original NDF framework \cite{sun2022topology} utilizes a composite loss function that generally includes:
    \begin{equation}
    \mathcal{L}_{\text{NDF}} = \mathcal{L}_{\text{SDF}} + \lambda_1 \mathcal{L}_{\text{diffeo}} + \lambda_2 \mathcal{L}_{\text{reg}},
    \end{equation}
    where $\mathcal{L}_{\text{SDF}} = \mathbb{E}_x[|f_\theta(z, \psi_z(x)) - \text{SDF}_{\text{target}}(x)|]$ measures how accurately the transformed SDF matches the target SDF values at sampled points, $\mathcal{L}_{\text{diffeo}} = \mathbb{E}_x[|\det(\nabla_x \psi_z) - 1|]$ enforces the diffeomorphic property by constraining the Jacobian determinant of the flow $\psi_z$ to remain positive and close to unity, and $\mathcal{L}_{\text{reg}}$ provides additional regularization on the flow field to ensure smoothness. This formulation directly connects the flow field $\psi_z$ with the SDF approximation task, ensuring that the transformed shapes maintain their topological structure.
    
    For our EIT application, we employ a pre-trained generative model that has already learned to generate topologically consistent shapes within its latent space, providing us with a learned shape prior. This allows us to navigate the latent space during reconstruction while inheriting the topological guarantees provided by the diffeomorphic constraints — a critical advantage for accurate interface representation in both medical imaging and industrial inspection applications where topological errors could lead to misdiagnosis or false detections.
    
    \subsection{Generative Prior-Guided Optimization on Learned Shape Manifolds}
    
    To address the severe ill-posedness, we transform the challenging shape optimization problem from the infinite-dimensional admissible set $\mathcal{A}$ into a principled search within the learned manifold of plausible shapes encoded in latent space $\mathcal{M}$ (Figure \ref{fig:latent_space}). This dimensionality reduction represents a paradigm shift in our approach to EIT interface reconstruction. Given a pre-trained 3D generative model, we seek the optimal latent code $z\in\mathcal{M}$ such that
    \begin{equation}\label{eq:latent_opt}
        \min_{z\in\mathcal{M}}\mathcal{L}(z) = J(\Gamma_z) := \frac{1}{2} \|u-f\|_{L^2(\Sigma)}^2,
    \end{equation}
    where $u$ solves the state equation \eqref{state1} with boundary $\Gamma_z$ which is implicitly defined as the zero level set of our neural network:
    \begin{equation*}
    \Gamma_z = \{ x\in \mathbb{R}^3| f_\theta(z,x)=0\}.
    \end{equation*}
    Following \cite{Younes2019}, we assume the regularity condition that $\nabla_x f_\theta(z,x)\neq 0$ whenever $f_\theta(z,x)=0$, ensuring well-defined surfaces with consistent normal vectors. This condition is naturally satisfied by properly trained neural SDF models that accurately approximate true signed distance functions.

    The pre-trained 3D generative model, trained on representative shape datasets, induces a low-dimensional manifold capturing the distribution of plausible geometries. Optimization in the latent space constitutes a principled search over this manifold to best explain the measurements. This provides exceptionally potent regularization that constrains solutions to be anatomically plausible, effectively addressing the ill-posedness inherent in EIT reconstruction.
    
    This latent representation fundamentally transforms the optimization landscape by dramatically reducing problem dimensionality — from the infinite-dimensional space of all possible interfaces to a compact latent space, typically of dimension $d \approx 256 \ll \infty$. Such dimensionality reduction not only enhances computational efficiency but also provides implicit regularization that helps mitigate the severe ill-posedness inherent in EIT reconstruction. The resulting optimization process can be interpreted as a projection of physics-based gradients onto the learned shape manifold, where physics dictates the descent direction while the generative prior constrains the path to anatomically plausible solutions.
    
    \begin{figure}[htb!]
        \centering
        \includegraphics[width=0.95\textwidth]{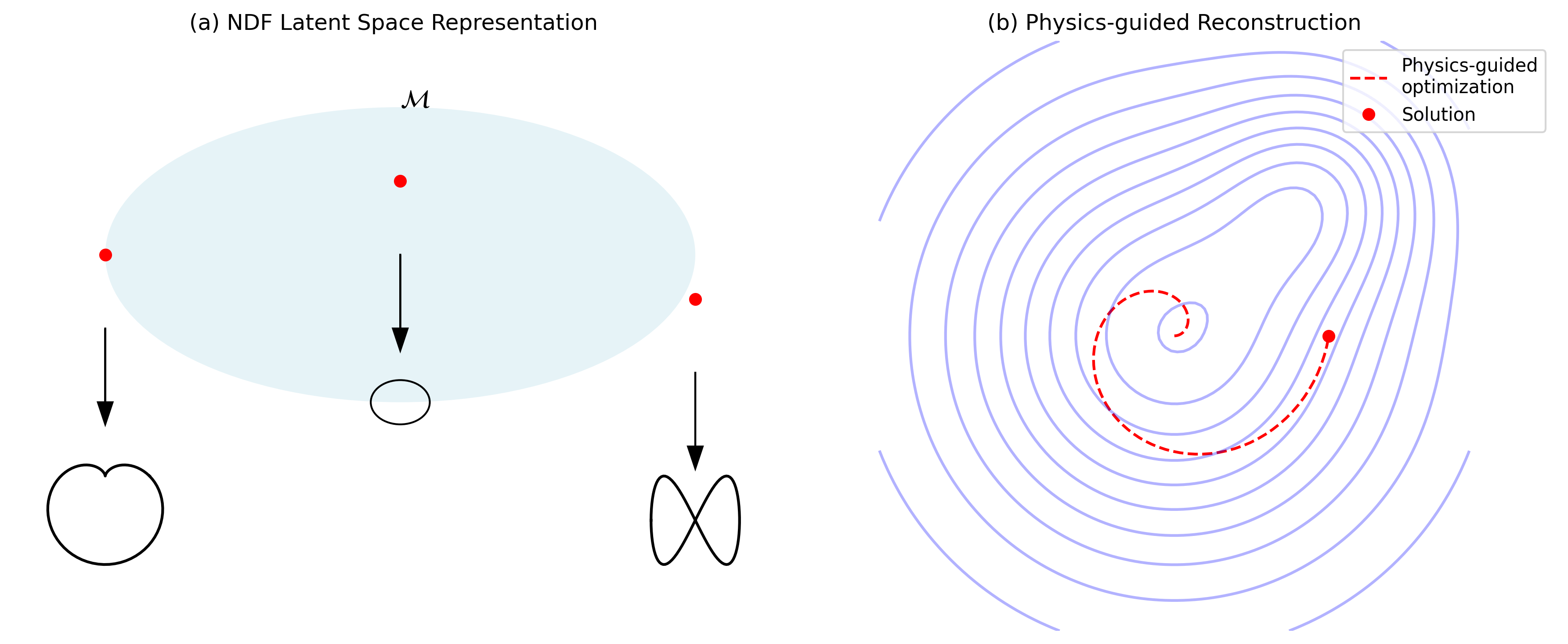}
        \caption{Integration of a pre-trained 3D generative prior with EIT reconstruction: (a) The generative model encodes plausible shapes into a low-dimensional manifold $\mathcal{M}$. (b) Our physics-informed optimization in $\mathcal{M}$ minimizes the EIT data misfit while constraining solutions to the learned manifold.}
        \label{fig:latent_space}
    \end{figure}
    
    \subsection{Adjoint-Based Gradients in Latent Space}
    
    To solve the optimization problem \eqref{eq:latent_opt} efficiently, we derive analytical expressions for the gradient of the objective function with respect to the latent variables. This requires extending classical shape calculus to optimization on a generative latent manifold. Let $V:\mathbb{R}\times\mathbb{R}^3\textcolor{black}{\rightarrow} \mathbb{R}^3$ be a velocity field and let $x(t,X)$ be the solution to
    
    \begin{equation}
    \frac{\mathrm{d}}{\mathrm{d}t}x(t,X)=V(t,x(t,X)), \quad x(0,X)=X.
    \end{equation}
    The transformation $T_t(X)=x(t,X)$ maps $\Gamma$ to $\Gamma_t=T_t(\Gamma)$. The shape derivative in direction $V$ is defined as:
    \begin{equation}
    dJ(\Gamma;V)=\lim_{t\rightarrow 0^+}\frac{J(\Gamma_t(V))-J(\Gamma)}{t}.
    \end{equation}
    By bridging this classical framework with our neural implicit representation, we derive the following key result:
    
    \begin{theorem}[Gradient Representation]\label{thm:gradient}
       \textcolor{black}{Let $\Omega_z$ be the domain bounded externally by $\Sigma$ and internally by $\Gamma_z$}, and assume that $\Omega_z$ is a bounded domain of class $C^2$. Let $\mathcal{L}(z)$ be defined in \eqref{eq:latent_opt}. If $\nabla_xf_\theta$ does not vanish on $\Gamma_z$, \textcolor{black}{and $n = -\frac{\nabla_x f_\theta}{\|\nabla_x f_\theta\|}$ denotes the outward normal to $\Omega_z$ on $\Gamma_z$}, then:
        \begin{equation}\label{eq:gradient1}
        \nabla \mathcal{L}(z)= \int_{\Gamma_z} \frac{\partial u}{\partial n} \frac{\partial w}{\partial n} \frac{\nabla_z f_\theta}{\|\nabla_x f_\theta\|} \mathrm{d}s,
        \end{equation}
        where $w$ solves the adjoint equation:
        \begin{equation}\label{eq:adjoint}
        \begin{cases} 
            \Delta w = 0 & \text{in } \Omega_z, \\
            w = 0 & \text{on } \Gamma_z, \\
            \partial_n w = u-f & \text{on } \Sigma.
        \end{cases}
        \end{equation}
    \end{theorem}
    
    \begin{proof}
    The proof follows from the chain rule and the shape derivative formula in \cite{chen2023solving}. Specifically, we use the relation between variations in the latent space and the corresponding shape variations through the implicit neural representation, applying the adjoint method to derive \eqref{eq:gradient1}.
    \end{proof}
    
    This gradient formulation provides a rare mathematical bridge between the physics-driven world of shape derivatives and the data-driven, probabilistic world of generative models. The integrand in \eqref{eq:gradient1} elegantly combines three critical components: (1) the sensitivity of the forward solution $\frac{\partial u}{\partial n}$, (2) the sensitivity of the adjoint solution $\frac{\partial w}{\partial n}$, and (3) the shape-to-latent mapping $\frac{\nabla_z f_\theta}{\|\nabla_x f_\theta\|}$, linking variations in data-misfit to perturbations in the learned shape manifold. This formulation provides a rigorous bridge between physics-driven shape derivatives and data-driven generative latent variables, enabling efficient navigation of the learned manifold toward solutions that match the measurements.
    
    Building on this gradient expression, we can update $z$ using stochastic gradient descent algorithms. In our experiments, we adopt the Adam optimizer, which efficiently combines gradient and momentum information, although other gradient-based methods could also be employed. The complete algorithm description is presented in Algorithm \ref{alg:main}. Our approach offers several key advantages, including efficient low-dimensional optimization enabled by the pre-trained generative prior. This combination enhances overall performance, particularly for challenging 3D high-contrast EIT reconstruction problems where maintaining correct topology is crucial for accurate interpretation.
    \begin{algorithm}[t]
    \caption{Generative Prior-Guided Neural Interface Reconstruction for 3D EIT}
    \label{alg:main}
    \begin{algorithmic}[1]
        \State \textbf{Input:} Voltage measurements, pre-trained 3D generative model $f_\theta$.
        \State \textbf{Initialize:} Latent code $z_0$, learning rates $\{\alpha_k\}$, max iterations $N$.
        \For{$k = 0$ to $N$}
            \State Generate interface $\Gamma_k = \{x | f_\theta(z_k,x)=0\}$ via marching cubes.
            \State Solve forward problem \eqref{state1} for $u_k$ on $\Omega_k$.
            \State Solve adjoint problem \eqref{eq:adjoint} for $w_k$.
            \State Generate an unbiased estimate $g_k$ of $\nabla \mathcal{L}(z_k)$ using the expression in \eqref{eq:gradient1}.
            \State Update $z$ using a stochastic gradient descent method with the estimate $g_k$. 
            \If{convergence criteria met}
                \State \textbf{break}
            \EndIf
        \EndFor
        \State \Return Reconstructed interface $\Gamma_N$.
    \end{algorithmic}
    \end{algorithm}
    
    \section{Convergence Analysis}\label{sec:analysis}
    
    We now establish the convergence properties of Algorithm \ref{alg:main} in the stochastic gradient descent (SGD) framework\cite{bottou2018optimization,JinZhouZou2020}. Given an initial guess $z_0$, the SGD algorithm updates the latent variable by 
    $$z_{k+1}= z_k-\alpha_k g_k, ~k=0,1,2, \cdots$$
    where $\alpha_k$ is the stepsize at iteration $k$ and $g_k$ is an unbiased estimate of $\nabla L(z_k)$. The key component of the convergence analysis is the shape derivative regularity estimate.
    
    First, we define the second-order shape derivative. Let $V_1,V_2:\mathbb{R}\times\mathbb{R}^3\mapsto\mathbb{R}^3$ be two differentiable velocity fields. Let the objective $J(\Gamma)$ be shape differentiable in the sense of Hadamard \cite{Delfour2011}. Then we define the second derivative by
    \begin{equation}
        d^2J(\Gamma;V_1;V_2)=\lim_{t\rightarrow 0^+}\frac{\textcolor{black}{dJ(\Gamma_t(V_2);V_1(t))}-dJ(\Gamma;V_1(0))}{t}.
    \end{equation}
    Following \cite[Theorem 6.2, Chapter 9]{Delfour2011}, we have
    \begin{equation}\label{eq:second_deriv}
        d^2J(\Gamma;V_1;V_2)=d^2J(\Gamma;V_1(0);V_2(0))+dJ(\Gamma;V_1'(0)),
    \end{equation}
    where
    \begin{align}
        d^2J(\Gamma;V_1(0);V_2(0))&:=\lim_{t\rightarrow 0^+}\frac{\textcolor{black}{dJ(T_t(\Gamma,V_2);V_1(0))-dJ(\Gamma;V_1(0))}}{t},\\
        dJ(\Gamma;V(0))&:= \lim_{t\rightarrow 0^+}\frac{J((T_t(\Gamma,V))-J(\Gamma)}{t},\\
        T_t(\Gamma,V)&=\{x+tV(0,x)|x\in\Gamma\},
    \end{align}
    and
    \begin{equation}
        V'(0)=\lim_{t\rightarrow0+}\frac{V(t,x)-V(0,x)}{t}.
    \end{equation}
    For a detailed discussion of second-order shape derivatives, we refer to \cite{Delfour2011,simon1989second} and \cite[Appendix]{hettlich1999second}.
    
    We begin with a key estimate on the shape derivatives.
    \begin{theorem}[Shape Derivative Bounds]\label{thm:hessian}
    Let $\Gamma$ be of class \textcolor{black}{$C^3$}. Then there exists $C>0$ independent of $V$ such that:
    \begin{align}
        |dJ(\Gamma;V(0))|&\le C\|V(0)\|_{C^1(\Gamma;\mathbb{R}^3)}, \quad\forall  V\in C([0,\varepsilon);C^1(\overline{\Omega},\mathbb{R}^3))\\
        |d^2J(\Gamma;V_1(0);V_2(0))|&\le C\|V_{1}(0)\|_{\textcolor{black}{C^{2}}(\Gamma,\mathbb{R}^3)}\|V_{2}(0)\|_{C^{2}(\Gamma,\mathbb{R}^3)},
    \end{align}
    for all $V_1,V_2\in C([0,\varepsilon);C^2(\overline{\Omega},\mathbb{R}^3))$, where $\varepsilon>0$ is sufficiently small.
    \end{theorem}
    
    The proof requires several preliminary results.
    
    \begin{lemma}[Energy Estimate]\label{lemma:energy} 
    \textcolor{black}{Let $\Omega \subset \mathbb{R}^n$ be a bounded Lipschitz domain representing the conducting region (i.e., $\Omega = D \setminus \overline{S}$), with its boundary explicitly partitioned into two disjoint parts: $\partial\Omega = \Gamma \cup \Sigma$. We consider the following auxiliary boundary value problem restricted strictly to the domain $\Omega$:}
    \begin{equation}
        \begin{cases} 
            \Delta u = 0 & \text{in } \Omega, \\
            u = f & \text{on } \Gamma, \\
            \partial_n u = 0 & \text{on } \Sigma, 
        \end{cases}
    \end{equation}
    where $f \in H^{1/2}(\Gamma)$. Then there exists $C = C(\Omega)$ such that
    \begin{equation}
        \|u\|_{H^1(\Omega)} \leq C\|f\|_{H^{1/2}(\Gamma)}.
    \end{equation}
    \end{lemma}
    \begin{proof}
    By integration by parts, we can easily get
    \begin{equation}
        \int_\Omega |\nabla u|^2 dx = \int_{\Gamma} f \partial_n u \, dS.
    \end{equation}
    Then using \textcolor{black}{Cauchy-Schwarz inequality and trace theorem}, we have
    \begin{equation}
        \left|\int_{\Gamma} f \partial_n u \, dS\right| \leq \|f\|_{H^{1/2}(\Gamma)} \|\partial_n u\|_{H^{-1/2}(\Gamma)} \leq C_1 \|f\|_{H^{1/2}(\Gamma)} \|u\|_{H^1(\Omega)}.
    \end{equation}
    \textcolor{black}{The generalized Poincaré inequality incorporating boundary trace norms\cite{evans2022partial} gives that}
    \begin{equation}
        \|u\|_{L^2(\Omega)} \leq C_2\left( \|\nabla u\|_{L^2(\Omega)} + \|f\|_{L^2(\Gamma)} \right).
    \end{equation}
    Combining these inequalities and applying Young's inequality yields the result.
    \end{proof}
    
    To prove Theorem \ref{thm:hessian}, we analyze the measurement operator $F:\mathcal{A} \rightarrow \mathbb{R}$:
    \begin{equation}
        F(\Gamma) = u(x_0)
    \end{equation}
    which maps an admissible boundary $\Gamma$ to the value of \textcolor{black}{the state function} $u$ at $x_0\in \Sigma$.
    By \cite[Theorem 2]{afraites2008second}, $F$ is twice continuously differentiable with
    \begin{equation}
        d^2 F(\Gamma;V_1(0);V_2(0))=u''(x_0),
    \end{equation}
    where $u''$ solves
    \begin{equation}\label{secondorder}
        \begin{cases}
            \Delta u^{\prime\prime}=0 & \text{in } \Omega,\\
            u^{\prime\prime}=\psi & \text{on } \Gamma,\\
            \partial_n u^{\prime \prime}=0 & \text{on } \Sigma,
        \end{cases}
    \end{equation}
    with \textcolor{black}{the boundary data}
    \begin{align}
        \psi =& (V_{1,n}V_{2,n}H-V_{1,\tau}\cdot (Dn V_{2,\tau}))\partial_n u-(V_{1,n}\partial_n u_2^\prime+V_{2,n}\partial_n u_1^\prime)\nonumber\\
        &+(V_{1,\tau} \cdot \nabla V_{2,n}+V_{2,\tau} \cdot \nabla V_{1,n})\partial_n u,
    \end{align}
     where $u$ is the solution of the state problem, and $n$ denotes the unit outward normal to the boundary $\Gamma$. The notation $V_{j,n} = V_j(0) \cdot n$ and \textcolor{black}{$V_{j,\tau} = V_j(0) - V_{j,n}n$} denote the normal scalar component and tangential vector projection of the velocity field $V_j(0)$, respectively. $H$ denotes the mean curvature of the boundary $\Gamma$, \textcolor{black}{and $Dn$ represents the derivative of the normal vector field on $\Gamma$}. Moreover, $u_j'$ for $j = 1,2$, is the solution to the following boundary value problem
    \begin{equation}\label{first order}
        \left\{ \begin{aligned}
            \Delta u_j^\prime&=0 \quad \text{ in } \Omega,\\
            u_j^\prime&=\textcolor{black}{-V_{j,n} \partial_n u} \quad\text{ on } \Gamma,\\
            \partial_n u_j^\prime&=0 \quad\text{ on } \Sigma.
        \end{aligned}\right.
    \end{equation}
    Still by \cite[Theorem 1]{afraites2008second}, we have
    \begin{equation}
        \textcolor{black}{d F(\Gamma;V_j(0))=u_j'(x_0), \quad j=1,2.}
    \end{equation}
     For the operator $F$, we have the following useful estimates.
    \begin{lemma}\label{shapebound}
        Under the assumption of Theorem \ref{thm:hessian}, the following estimates hold
        \begin{align}
            |dF(\Gamma;V(0))|&\le C\|V(0)\|_{C^1(\Gamma;\mathbb{R}^3)}, \quad\forall  V\in C([0,\varepsilon);C^1(\overline{\Omega},\mathbb{R}^3)),\\
            |d^2F(\Gamma;V_1(0);V_2(0))|&\le C\big\|V_{1}(0)\|_{\textcolor{black}{C^{2}}(\Gamma,\mathbb{R}^3)}\|V_{2}(0)\|_{C^{2}(\Gamma,\mathbb{R}^3)},\,\quad \forall V_1,V_2\in C([0,\varepsilon);C^2(\overline{\Omega},\mathbb{R}^3)).
        \end{align}
    \end{lemma}
    \begin{proof}
        First, the $C^3$ regularity of the boundary $\Gamma$ implies that, for a bounded neighborhood $\Omega^+$ of $\Gamma$ within $\Omega$,
        \begin{equation}\label{eqn:u-apriori}
            u\in H^3(\Omega^+),\quad u|_{\Gamma} \in H^{5/2}(\Gamma) \quad \mbox{and}\quad \partial_n u \in H^{3/2}(\Gamma).
        \end{equation}
        By \eqref{first order}, we have 
        \begin{align}
            |dF(\Gamma;V(0))|=|u'(x_0)|\le \max_{x \in \Sigma}|u'(x)| = \|u'\|_{C^0(\Sigma)}.
        \end{align}
        
        We consider a neighborhood of $\Sigma$ within $\Omega$, denoted as $\Omega'$, such that $\text{dist}(\Omega', \Gamma) > 0$. By the classical elliptic regularity theory (see \cite{evans2022partial}) and the trace theorem, there exists a constant $C>0$ such that
        \begin{equation}\label{local boundary regularity}
        \|u'\|_{C^0(\Sigma)} \leq C\|u'\|_{H^{3/2}(\Sigma)} \leq C\|u'\|_{H^2(\Omega')} \leq C\|u'\|_{H^1(\Omega)},
        \end{equation}
        \textcolor{black}{where we have used local boundary elliptic regularity. Since $u'$ is harmonic in $\Omega$ and satisfies the homogeneous Neumann boundary condition $\partial_n u' = 0$ on the smooth boundary $\Sigma$, the regularity is elevated locally near the boundary, guaranteeing that $u' \in H^2(\Omega')$ given $u' \in H^1(\Omega)$.}
        
        \textcolor{black}{Next, we estimate $\|u'\|_{H^1(\Omega)}$. By the variational theory of elliptic equations, Lemma \ref{lemma:energy}, and the fact that $u'$ satisfies equation \eqref{first order} with the boundary condition $u'|_{\Gamma} = -V_n \partial_n u$, we obtain}
        \begin{equation}
        \|u'\|_{H^1(\Omega)} \leq C\|u'\|_{H^{1/2}(\Gamma)} = C\|-V_n \partial_n u\|_{H^{1/2}(\Gamma)} = C\|V_n \partial_n u\|_{H^{1/2}(\Gamma)},
        \end{equation}
        where $C$ is a constant depending only on the domain $\Omega$.
        
        By the product rule in Sobolev spaces, when $V_n \in C^1(\Gamma)$ and $\partial_n u \in H^{1/2}(\Gamma)$, we have
        \begin{equation}
        \|V_n \partial_n u\|_{H^{1/2}(\Gamma)} \leq C\|V_n\|_{C^1(\Gamma)}\|\partial_n u\|_{H^{1/2}(\Gamma)}.
        \end{equation}
        
        Since under the regularity assumption \eqref{eqn:u-apriori}, $\partial_n u \in \textcolor{black}{H^{3/2}}(\Gamma) \hookrightarrow H^{1/2}(\Gamma)$ and the norm is bounded, i.e., there exists a constant $C>0$, independent of $V(0)$, such that $\|\partial_n u\|_{H^{1/2}(\Gamma)} \leq C$, we obtain
        \begin{equation}
        \|V_n \partial_n u\|_{H^{1/2}(\Gamma)} \leq C\|V_n\|_{C^1(\Gamma)}.
        \end{equation}
        
        Combining the above inequalities, \textcolor{black}{especially the local boundary regularity\eqref{local boundary regularity}, we have:} 
        \begin{equation}
        \max_{x\in \Sigma}|u'(x)| \leq C\|u'\|_{C^0(\Sigma)} \leq C\|u'\|_{H^1(\Omega)} \leq C\|u'|_{\Gamma}\|_{H^{1/2}(\Gamma)} \leq C\|V_n\|_{C^1(\Gamma)},
        \end{equation}
        thus
        \begin{equation}
        |dF(\Gamma;V(0))| \leq \max_{x\in \Sigma}|u'(x)| \leq C\|V_n\|_{C^1(\Gamma)} \leq C\|V(0)\|_{C^1(\Gamma;\mathbb{R}^3)}.
        \end{equation}
        This shows the first assertion. 

        {\color{black}
        Next, from \eqref{secondorder}, by applying the exact same local boundary regularity and embedding arguments (cf. \eqref{local boundary regularity}) as we did for $u'$, we establish the continuous chain of bounds for $u''$:
        \begin{equation}
            \begin{aligned}
                |d^2 F(\Gamma;V_1(0);V_2(0))| =& |u''(x_0)| \le \|u''\|_{C^0(\Sigma)} \\
                \le & C\|u''\|_{H^{1/2}(\Gamma)} = C\|\psi\|_{H^{1/2}(\Gamma)} \\
                \le & C \big( \|(V_{1,n}V_{2,n}H-V_{1,\tau}\cdot (Dn V_{2,\tau}))\partial_n u\|_{H^{1/2}(\Gamma)} \\
                & \quad + \|V_{1,n}\partial_n u_2^\prime\|_{H^{1/2}(\Gamma)} + \|V_{2,n}\partial_n u_1^\prime\|_{H^{1/2}(\Gamma)} \\
                & \quad + \|(V_{1,\tau} \cdot \nabla V_{2,n})\partial_n u\|_{H^{1/2}(\Gamma)} + \|(V_{2,\tau} \cdot \nabla V_{1,n})\partial_n u\|_{H^{1/2}(\Gamma)} \big) \\
                =& C(I_1+I_2+I_3+I_4+I_5).
            \end{aligned}
        \end{equation}
        }
        
        Then we can prove the second assertion term by term. For the first term $I_1$,
        \begin{align}
            I_1 =&\|(V_{1,n}V_{2,n}H-V_{1,\tau}\cdot (Dn V_{2,\tau}))\partial_n u\|_{H^{1/2}(\Gamma)}\\
            \le &C\|( |V_{1,n}V_{2,n}|+|V_{1,\tau}| |V_{2,\tau}| )\partial_n u\|_{H^{1/2}(\Gamma)},
        \end{align}
        where we used the fact that the shape operator $Dn$ is bounded on the smooth boundary $\Gamma$.
        
        Using the product rule for Sobolev spaces, for functions $f, g$ where $f \in H^{1/2}(\Gamma)$ and $g \in C^1(\Gamma)$, we have:
        \begin{equation}
        \|fg\|_{H^{1/2}(\Gamma)} \leq C\|f\|_{H^{1/2}(\Gamma)}\|g\|_{C^1(\Gamma)}
        \end{equation}
        
        Applying this to our case, and noting that \textcolor{black}{$\partial_n u \in H^{3/2}(\Gamma) \hookrightarrow H^{1/2}(\Gamma)$} with bounded norm due to \eqref{eqn:u-apriori}, we get:
        \begin{align}
            \|(|V_{1,n}V_{2,n}|+|V_{1,\tau}| |V_{2,\tau}|)\partial_n u\|_{H^{1/2}(\Gamma)} &\leq C\||V_{1,n}V_{2,n}|+|V_{1,\tau}| |V_{2,\tau}|\|_{C^1(\Gamma)}\|\partial_n u\|_{H^{1/2}(\Gamma)}\\
            &\leq C\||V_{1,n}V_{2,n}|+|V_{1,\tau}| |V_{2,\tau}|\|_{C^1(\Gamma)}
        \end{align}
        
        For products of functions in $C^1$, we have:
        \begin{align}
            \||V_{1,n}V_{2,n}|+|V_{1,\tau}| |V_{2,\tau}|\|_{C^1(\Gamma)} &\leq \|V_{1,n}V_{2,n}\|_{C^1(\Gamma)} + \||V_{1,\tau}| |V_{2,\tau}|\|_{C^1(\Gamma)}\\
            &\leq \|V_{1,n}\|_{C^1(\Gamma)}\|V_{2,n}\|_{C^1(\Gamma)} + \|V_{1,\tau}\|_{C^1(\Gamma)}\|V_{2,\tau}\|_{C^1(\Gamma)}
        \end{align}
        
        The components $V_{j,n}$ and $V_{j,\tau}$ are projections of $V_j$ onto the normal and tangential directions, respectively. Specifically, the tangential projection is given by $V_{j,\tau} = V_j - V_{j,n}n$. This projection relies entirely on the geometry of the boundary through the normal vector $n$. Since the boundary $\Gamma$ is $C^3$, the normal vector field is at least $C^2$. When we take derivatives of this projection, we only need to account for the derivatives of the vector field and the normal vector $n$.
        
        For a vector field $V_j \in C^1(\Gamma;\mathbb{R}^3)$, its normal component $V_{j,n} = V_j \cdot n$ satisfies:
        \begin{equation}
        \|V_{j,n}\|_{C^1(\Gamma)} \leq C\|V_j\|_{C^1(\Gamma;\mathbb{R}^3)}.
        \end{equation}
        \textcolor{black}{This holds because computing the surface gradient $\nabla_\Gamma (V_j \cdot n) = (\nabla_\Gamma V_j)n + (\nabla_\Gamma n)V_j$ involves only the first-order derivatives of the vector field $V_j$ and the first-order derivatives of the normal vector $n$ (which requires the boundary $\Gamma$ to be at least $C^2$).}
        
        Similarly, for the tangential component:
        \begin{equation}
        \|V_{j,\tau}\|_{C^1(\Gamma)} \leq C\|V_j\|_{C^1(\Gamma;\mathbb{R}^3)}.
        \end{equation}
        
        Therefore, utilizing the fact that $\| \cdot \|_{C^1} \le \| \cdot \|_{C^2}$, we can bound the product as:
        \begin{align}
            \||V_{1,n}V_{2,n}|+|V_{1,\tau}| |V_{2,\tau}|\|_{C^1(\Gamma)} &\leq C\|V_1\|_{C^1(\Gamma;\mathbb{R}^3)}\|V_2\|_{C^1(\Gamma;\mathbb{R}^3)} \nonumber \\
            &\leq C\|V_1\|_{C^2(\Gamma;\mathbb{R}^3)}\|V_2\|_{C^2(\Gamma;\mathbb{R}^3)}.
        \end{align}
        
        For the term $I_1$, we now have:
        \begin{equation}
        I_1 \leq C\|V_1\|_{C^2(\Gamma;\mathbb{R}^3)}\|V_2\|_{C^2(\Gamma;\mathbb{R}^3)}
        \end{equation}
        
        \textcolor{black}{The terms $I_2$ through $I_5$ involve at most first-order derivatives of the vector fields and the normal vector. Bounding their $H^{1/2}(\Gamma)$ norms via the product rule requires one additional derivative. Consequently, $C^2$ regularity for the vector fields and $C^3$ regularity for the boundary $\Gamma$ are sufficient to yield:}
        \begin{equation}
        I_j \leq C\|V_1\|_{C^2(\Gamma;\mathbb{R}^3)}\|V_2\|_{C^2(\Gamma;\mathbb{R}^3)}, \quad j = 2,3,4,5.
        \end{equation}
        
        Combining all these estimates, we conclude:
        \begin{equation}
        |d^2 F(\Gamma;V_1(0);V_2(0))| \leq C\|V_1\|_{C^2(\Gamma;\mathbb{R}^3)}\|V_2\|_{C^2(\Gamma;\mathbb{R}^3)}.
        \end{equation}
        
        This completes the proof of the second assertion.
    \end{proof}
    \begin{proof}[Proof of Theorem \ref{thm:hessian}]
        With the help of Lemma \ref{shapebound}, we can prove Theorem \ref{thm:hessian} similarly to \cite{chen2023solving}, here we omit the details.
    \end{proof}
    
    Now, we can give the convergence of the SGD algorithm. 
    \begin{theorem}\label{SGDalg}
        Let \(\{z_k\}_{k=0}^\infty \subset \mathbb{R}^d\) be a sequence of latent codes generated by Algorithm \ref{alg:main}, where each $z_k$ implicitly represents a shape through the neural network $f_\theta(z_k, \cdot)$. Assume that this sequence is uniformly bounded. Let the step size \(\alpha_k\) be fixed, such that \(\alpha_k = \bar{\alpha}\) for all \(k \in \mathbb{N}\). We also consider the shape derivatives approximated by stochastic gradients $g_k$ in the latent space, \textcolor{black}{with the following properties conditionally on $z_k$:}
        \begin{itemize}
            \item[(i)] There exist scalars \(\mu_G \geq \mu > 0\) such that for all \(k \in \mathbb{N}\):
            \begin{equation}
                \nabla \mathcal{L}(z_k)^T \mathbb{E}[\textcolor{black}{g_k|z_k}] \geq \mu \|\nabla \mathcal{L}(z_k)\|_2^2 \quad \text{and} \quad \|\mathbb{E}[g_k|z_k]\|_2 \leq \mu_G \|\nabla \mathcal{L}(z_k)\|_2.
            \end{equation}
            
            \item[(ii)] There exist scalars \(M \geq 0\) and \(M_V \geq 0\) such that for all \(k \in \mathbb{N}\):
            \begin{equation}
                \mathbb{V}[\textcolor{black}{g_k|z_k}] \leq M + M_V \|\nabla \mathcal{L}(z_k)\|_2^2,
            \end{equation}
            where $\mathbb{V}[g_k|z_k]$ denotes the conditional variance of the stochastic gradient $g_k$ given $z_k$.
        \end{itemize}

        Then, \textcolor{black}{if the step size satisfies \(0 < \bar{\alpha} \leq \frac{\mu}{L(\mu_G^2 + M_V)}\)}, the following estimate holds:
        \begin{equation}
            \mathbb{E} \left[ \frac{1}{K} \sum_{k=0}^{\textcolor{black}{K-1}} \|\nabla \mathcal{L}(z_k)\|_2^2 \right] \leq \frac{\bar{\alpha} L M}{\mu} + \frac{2(\mathcal{L}(z_0) - \mathcal{L}_{\inf})}{K \mu \bar{\alpha}} \xrightarrow{K \to \infty} \frac{\bar{\alpha} L M}{\mu},
        \end{equation}
        where $\mathcal{L}_{\inf}$ is a lower bound on $\mathcal{L}(z)$.
    \end{theorem}
    
    \begin{remark}\label{rmk:gk_explanation}
        \color{black}{The assumptions on the stochastic gradient $g_k$ are directly motivated by our specific numerical implementation. In Algorithm \ref{alg:main}, $g_k$ is computed by sampling points on the interface $\Gamma_{z_k}$ to approximate the shape derivative integral. This sampling strategy, combined with the boundary integral equation (BIE) solutions for the state and adjoint variables, provides an unbiased estimate of the true gradient given the current latent code $z_k$, while inherently introducing a bounded variance.}
    \end{remark}

    \begin{proof}
        Our proof consists of two key parts: first establishing how the properties of the neural shape representation affect the Lipschitz continuity of the objective gradient, and then analyzing the convergence of the resulting stochastic optimization process in the latent space.
        
        \textbf{Step 1: Establishing Lipschitz Continuity}
        
        From Theorem \ref{thm:hessian}, we know that the second-order shape derivative is bounded
        \begin{equation}
            |d^2J(\Gamma;V_1(0);V_2(0))| \leq C\|V_1(0)\|_{C^2(\Gamma;\mathbb{R}^3)}\|V_2(0)\|_{C^2(\Gamma;\mathbb{R}^3)}.
        \end{equation}
        
        In our implicit neural representation, perturbations in the latent space $z$ induce shape perturbations through the zero level set of $f_\theta(z,x)$. Using the chain rule and the structure of our neural network, we can establish that
        \begin{equation}
            \|\nabla^2 \mathcal{L}(z)\|_2 \leq L,
        \end{equation}
        where $L$ depends on the shape derivative bounds from Theorem \ref{thm:hessian} and the properties of our neural network $f_\theta$.
        
        This yields the Lipschitz continuity of the gradient in the latent space:
        \begin{equation}
            \|\nabla \mathcal{L}(z + \Delta z) - \nabla \mathcal{L}(z)\|_2 \leq L\|\Delta z\|_2.
        \end{equation}
        
        \textbf{Step 2: Convergence Analysis}
        
        With the Lipschitz continuity established, we now analyze the convergence of our algorithm. Following the approach of \cite{bottou2018optimization} but adapted to our context, we proceed as follows.
        
        The Lipschitz continuity implies
        \begin{equation}
            \mathcal{L}(z + \Delta z) \leq \mathcal{L}(z) + \nabla\mathcal{L}(z)^T\Delta z + \frac{L}{2}\|\Delta z\|_2^2.
        \end{equation}
        
        For our update $z_{k+1} = z_k - \bar{\alpha}g_k$, where $g_k$ approximates the shape derivative in the latent space, taking conditional expectation, we have
        \begin{align}
            \mathbb{E}[\mathcal{L}(z_{k+1})|z_k] &\leq \mathcal{L}(z_k) - \bar{\alpha}\nabla\mathcal{L}(z_k)^T\mathbb{E}[g_k|z_k] + \frac{L\bar{\alpha}^2}{2}\mathbb{E}[\|g_k\|_2^2|z_k].
        \end{align}
        
        \textcolor{black}{Based on the assumptions stated in Theorem \ref{SGDalg} regarding the conditional expectation and variance of the stochastic gradient $g_k$, we obtain the following properties:}
        \begin{align}
            \nabla\mathcal{L}(z_k)^T\mathbb{E}[g_k|z_k] &\geq \mu\|\nabla\mathcal{L}(z_k)\|_2^2\\
            \mathbb{E}[\|g_k\|_2^2|z_k] &= \|\mathbb{E}[g_k|z_k]\|_2^2 + \mathbb{V}[g_k|z_k]\\
            &\leq \mu_G^2\|\nabla\mathcal{L}(z_k)\|_2^2 + M + M_V\|\nabla\mathcal{L}(z_k)\|_2^2\\
            &= (\mu_G^2 + M_V)\|\nabla\mathcal{L}(z_k)\|_2^2 + M
        \end{align}
        
        Substituting back:
        \begin{align}
            \mathbb{E}[\mathcal{L}(z_{k+1})|z_k] &\leq \mathcal{L}(z_k) - \bar{\alpha}\mu\|\nabla\mathcal{L}(z_k)\|_2^2 + \frac{L\bar{\alpha}^2}{2}((\mu_G^2 + M_V)\|\nabla\mathcal{L}(z_k)\|_2^2 + M)\\
            &= \mathcal{L}(z_k) - \bar{\alpha}\left(\mu - \frac{L\bar{\alpha}(\mu_G^2 + M_V)}{2}\right)\|\nabla\mathcal{L}(z_k)\|_2^2 + \frac{L\bar{\alpha}^2 M}{2}
        \end{align}
        
        \textcolor{black}{Using the step size condition $\bar{\alpha} \leq \frac{\mu}{L(\mu_G^2 + M_V)}$ assumed in Theorem \ref{SGDalg}}, we have 
        \begin{equation}
        \mu - \frac{L\bar{\alpha}(\mu_G^2 + M_V)}{2} \geq \frac{\mu}{2}.
        \end{equation}
        
        This ensures:
        \begin{equation}
            \mathbb{E}[\mathcal{L}(z_{k+1})|z_k] \leq \mathcal{L}(z_k) - \frac{\bar{\alpha}\mu}{2}\|\nabla\mathcal{L}(z_k)\|_2^2 + \frac{L\bar{\alpha}^2 M}{2}.
        \end{equation}
        
        Taking total expectation and rearranging:
        \begin{align}
            \frac{\bar{\alpha}\mu}{2}\mathbb{E}[\|\nabla\mathcal{L}(z_k)\|_2^2] &\leq \mathbb{E}[\mathcal{L}(z_k)] - \mathbb{E}[\mathcal{L}(z_{k+1})] + \frac{L\bar{\alpha}^2 M}{2}
        \end{align}
        
        Summing over $k = 0, 1, \ldots, K-1$:
        \begin{align}
            \frac{\bar{\alpha}\mu}{2}\sum_{k=0}^{K-1}\mathbb{E}[\|\nabla\mathcal{L}(z_k)\|_2^2] &\leq \mathbb{E}[\mathcal{L}(z_0)] - \mathbb{E}[\mathcal{L}(z_{K})] + \frac{KL\bar{\alpha}^2 M}{2}\\
            &\leq \mathcal{L}(z_0) - \mathcal{L}_{\inf} + \frac{KL\bar{\alpha}^2 M}{2},
        \end{align}
        \textcolor{black}{where we used the fact that the initial point $z_0$ is deterministic, hence $\mathbb{E}[\mathcal{L}(z_0)] = \mathcal{L}(z_0)$.}
        
        Dividing by $K$ and by $\frac{\bar{\alpha}\mu}{2}$:
        \begin{align}
            \mathbb{E}\left[\frac{1}{K}\sum_{k=0}^{K-1}\|\nabla\mathcal{L}(z_k)\|_2^2\right] &\leq \frac{2(\mathcal{L}(z_0) - \mathcal{L}_{\inf})}{K\bar{\alpha}\mu} + \frac{\bar{\alpha}LM}{\mu}
        \end{align}
        
        As $K \to \infty$:
        \begin{align}
            \lim_{K\to\infty}\mathbb{E}\left[\frac{1}{K}\sum_{k=0}^{K-1}\|\nabla\mathcal{L}(z_k)\|_2^2\right] &\leq \frac{\bar{\alpha}LM}{\mu}
        \end{align}
        This completes the proof, establishing convergence of our implicit neural shape optimization algorithm in the average gradient sense, with a bound that reflects both the shape optimization aspects and the neural network representation characteristics of our approach.
    \end{proof}

    \begin{remark}
    Theorem \ref{SGDalg} establishes convergence guarantees for the SGD method with fixed step size. These results can be extended to the Adam algorithm, which is used in our implementation (Algorithm \ref{alg:main}). Adam combines momentum and adaptive step sizes, requiring additional analysis of the momentum accumulation and second moment estimation. Following the approach in \cite{kingma2014adam} and \cite{reddi2019convergence}, one can show that, under similar assumptions but with appropriate modifications to account for the momentum terms and adaptive step sizes, Adam converges with comparable guarantees. Specifically, the key modification involves analyzing how the momentum terms affect the descent direction alignment and how the adaptive scaling affects the effective step size. The extended analysis would yield similar asymptotic behavior but with potentially improved constants due to the adaptive nature of the algorithm.
    \end{remark}
    
    \section{Numerical Results}\label{sec:results}
    We present comprehensive numerical experiments to validate the effectiveness of our implicit neural shape optimization approach to high-contrast EIT reconstruction in Algorithm \ref{alg:main}. The experiments are designed to evaluate both the accuracy of the reconstruction and the robustness of the algorithm in various challenging scenarios.
    
    \subsection{Implementation Details}
   The outer boundary $\Sigma$ is defined as the sphere of radius 1.5. For discretization, we employ a regular grid with a spacing of 0.06 for the marching cube algorithm. Our implementation uses the BEMPP-CL package\cite{betcke2021bempp} to solve boundary integral equations in the forward problem. The optimization process relies on the Adam algorithm with an initial learning rate of 0.01, operating in a latent space of dimension 256 for the representation of the shape.
    
    To simulate realistic measurement conditions, we introduce Gaussian noise into the boundary measurement data:
    \begin{equation}
        f^\delta = (1+\delta \xi_{ml})f_{\text{exact}}, \quad \xi_{ml} \sim \mathcal{N}(0,1),
    \end{equation}
    where $\delta$ controls the noise level.
    
    We employ three complementary metrics to quantitatively assess reconstruction quality:
    \begin{enumerate}
        \item \textbf{Indicator Error:} Measures the volumetric difference between the reconstructed and target shapes:
        \begin{equation}
            e(\hat{S}) = \sum_{x_i}\|\hat{S}(x_i)-S^\dag(x_i)\|_2,
        \end{equation}
        where $\hat{S}$ and $S^\dag$ are the characteristic functions of the reconstructed and target shapes, respectively, \textcolor{black}{and $\{x_i\}$ denotes the set of uniformly sampled grid points within the computational domain}.
        \item \textbf{Hausdorff Distance:} Captures the maximum geometric deviation between surfaces:
        \begin{equation}
            d_H(S_1, S_2) = \max\left\{\sup_{a \in S_1} \inf_{b \in S_2} \|a - b\|, \sup_{b \in S_2} \inf_{a \in S_1} \|b - a\|\right\}.
        \end{equation}
        \item \textbf{Volume Difference:} Provides a global measure of shape accuracy:
        \begin{equation}
            d_V(S_1, S_2) = |V_1 - V_2|,
        \end{equation}
        where $V_1$ and $V_2$ denote the volumes enclosed by surfaces $S_1$ and $S_2$ respectively.
    \end{enumerate}

    \subsection{Experimental Results}
    We evaluate our method on two distinct scenarios to demonstrate its versatility and effectiveness in different reconstruction contexts. These scenarios are carefully designed to assess both the accuracy and robustness of our algorithm under varying geometric conditions. The dataset used is Abdomen1k dataset\cite{9497733}. We randomly select training samples from the dataset and evaluate on held-out test cases.
    
    \subsubsection{Scenario I: Pancreas Reconstruction}
    The first scenario focuses on the pancreas model. Two examples are included in this scenario. Figures \ref{example1:result} - \ref{noise:error} are for the first example and Figures \ref{example12:result}-\ref{example12:intermidiate} are for the second example.
    
    In Figure \ref{example1:result}, the first row shows the initial guess for Algorithm \ref{alg:main}, the second row the final reconstructions, and the third row the reconstruction target. The reconstructions are shown at four different view angles: perspective view (first column), superior view (second column), anterior view (third column), and lateral view (fourth column). Each perspective reveals different anatomical features. This example demonstrates the ability of the algorithm to reconstruct complex pancreas structures. The convergence of the algorithm versus the iteration number is shown in Figure \ref{example1:error}, where we observe a rapid initial decrease in the loss function followed by gradual refinement, indicating efficient optimization in the latent space. The indicator error, Hausdorff distance, and volume distance metrics show consistent improvement throughout the optimization process, confirming the progressive geometric convergence towards the target shape.
    
    \begin{figure}[htbp]
        \centering
        \includegraphics[width=.8\textwidth]{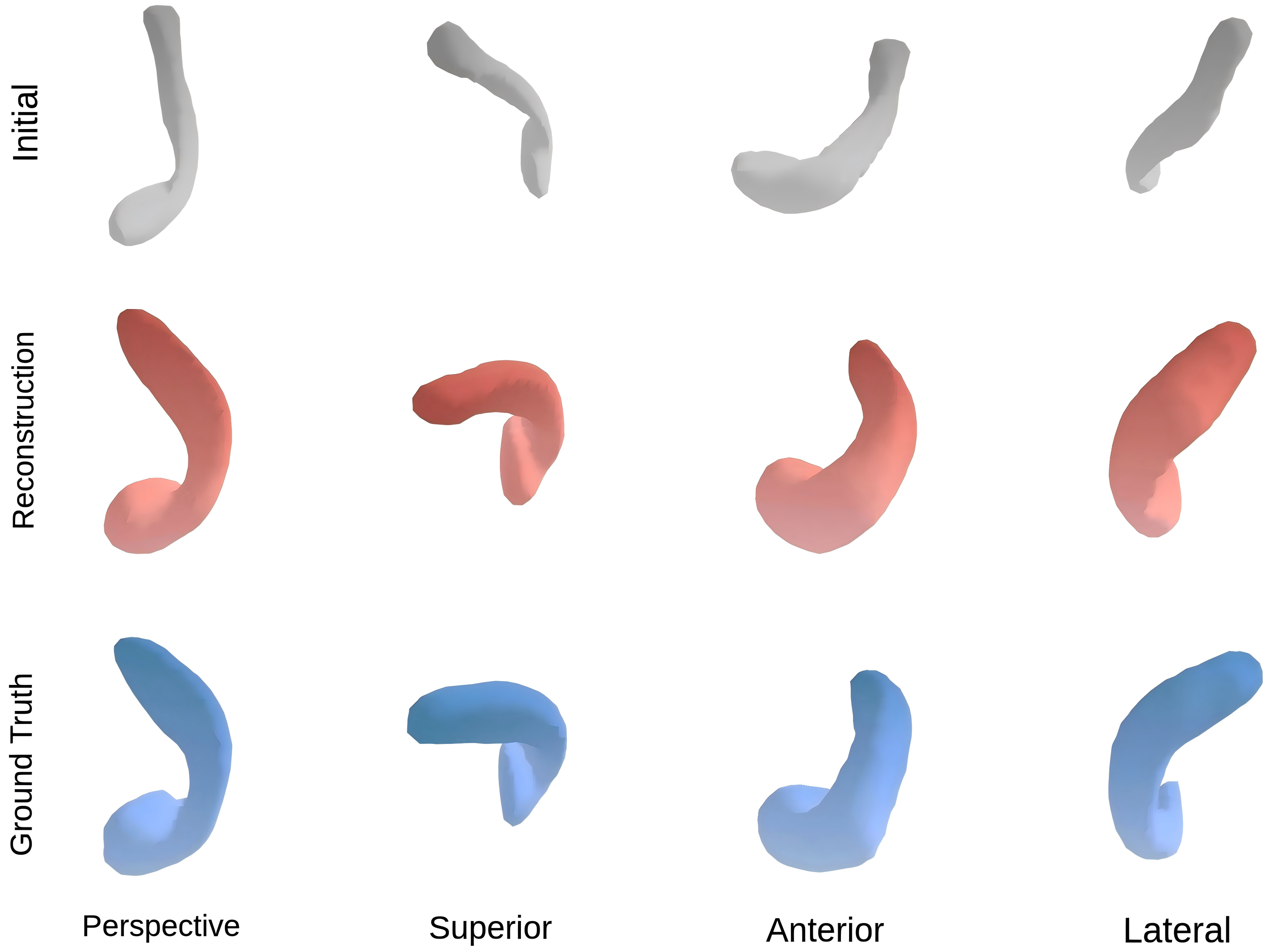}
        \caption{The reconstructions for Scenario I: the first, second, and third rows refer to the initial, the optimal recovery, and the ground truth target, respectively.}
        \label{example1:result}
    \end{figure}
    
    \begin{figure}[htbp]
        \centering
        \includegraphics[width=0.6\textwidth]{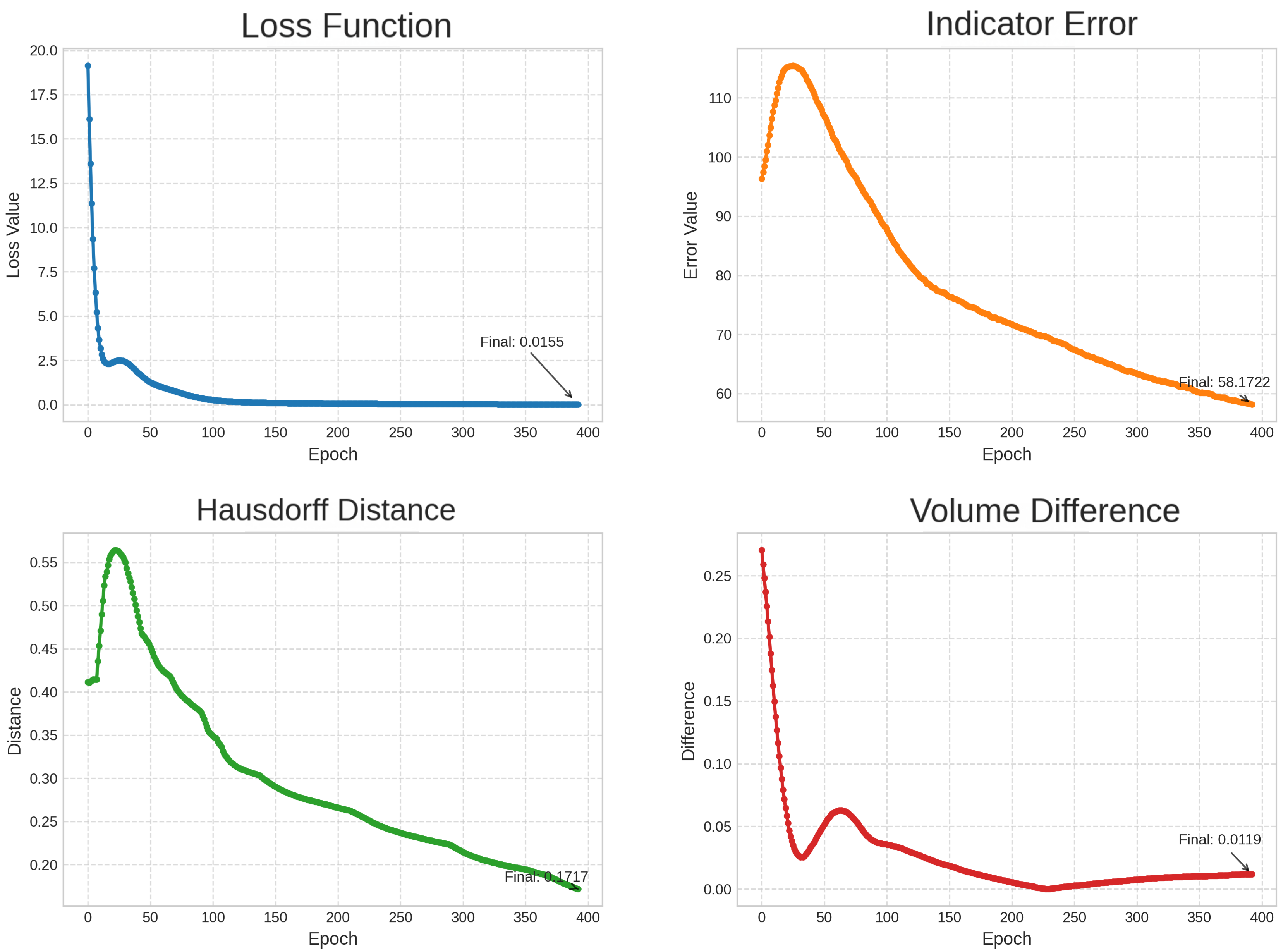}
        \caption{Convergence of the algorithm for Scenario I in terms of the loss $\mathcal{L}(z)$ (upper left), the indicator error $e$ (\textcolor{black}{upper right}), the Hausdorff error (lower left) and the volume error (lower right).}
        \label{example1:error}
    \end{figure}
    
    To evaluate the stability of our method under realistic measurement conditions, we examined the reconstruction performance with significant noise (20\%) added to the boundary measurement data. Figure \ref{noise:result} presents the reconstruction results with noisy measurements. Despite the considerable measurement noise, our method successfully recovers the core geometric features of the target shape, preserving both the general curved structure and the approximate proportions. The convergence metrics in Figure \ref{noise:error} further confirm the robustness of our algorithm.
    \begin{figure}[htbp]
        \centering
        \includegraphics[width=.8\textwidth]{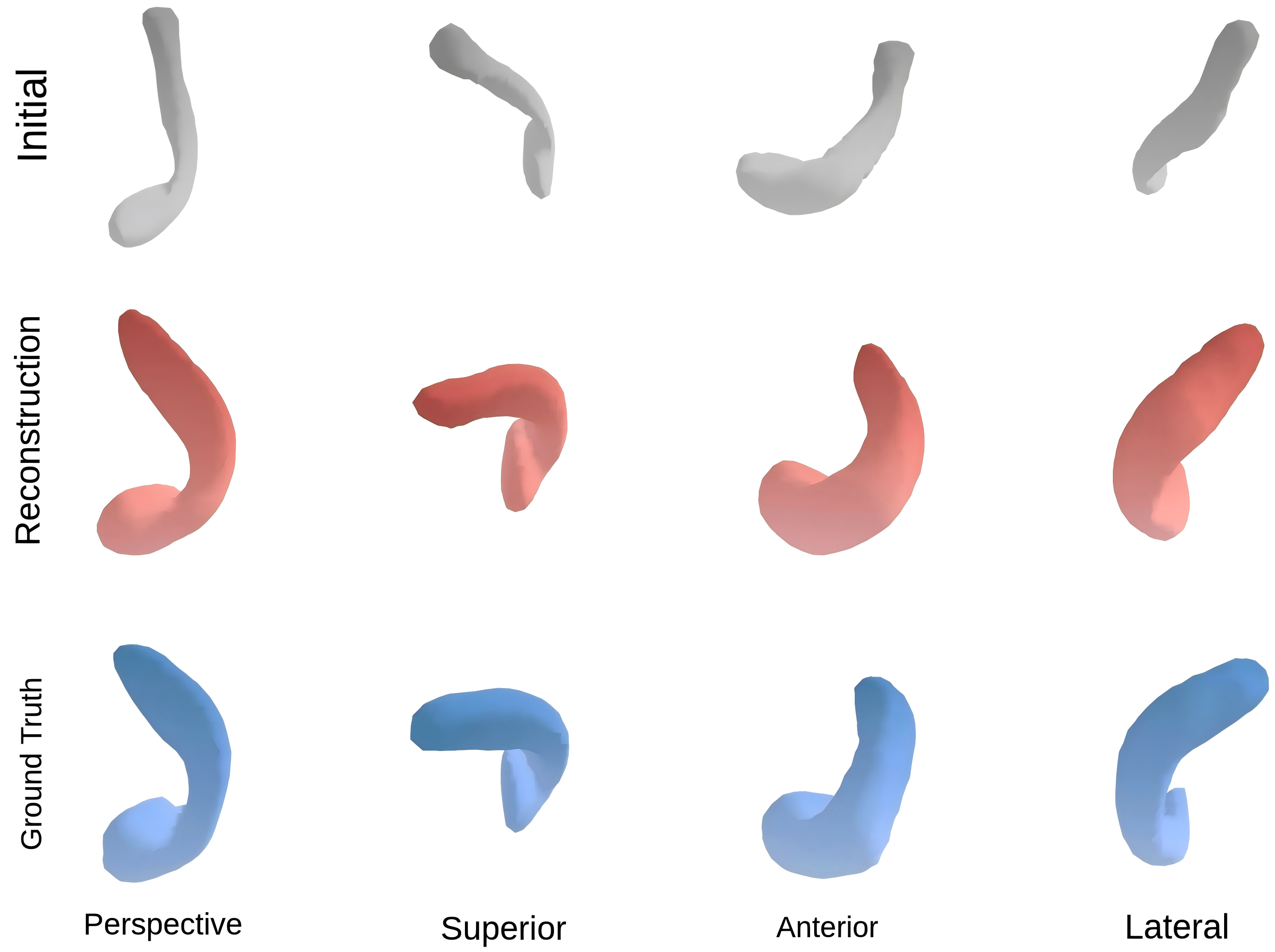}
        \caption{Reconstruction results with 20\% noise in measurement data: initial guess (first row), final reconstruction (middle row), and ground truth (third row). The different columns show different viewing angles of the same reconstruction.}
        \label{noise:result}
    \end{figure}
    \begin{figure}[htbp]
        \centering
        \includegraphics[width=0.6\textwidth]{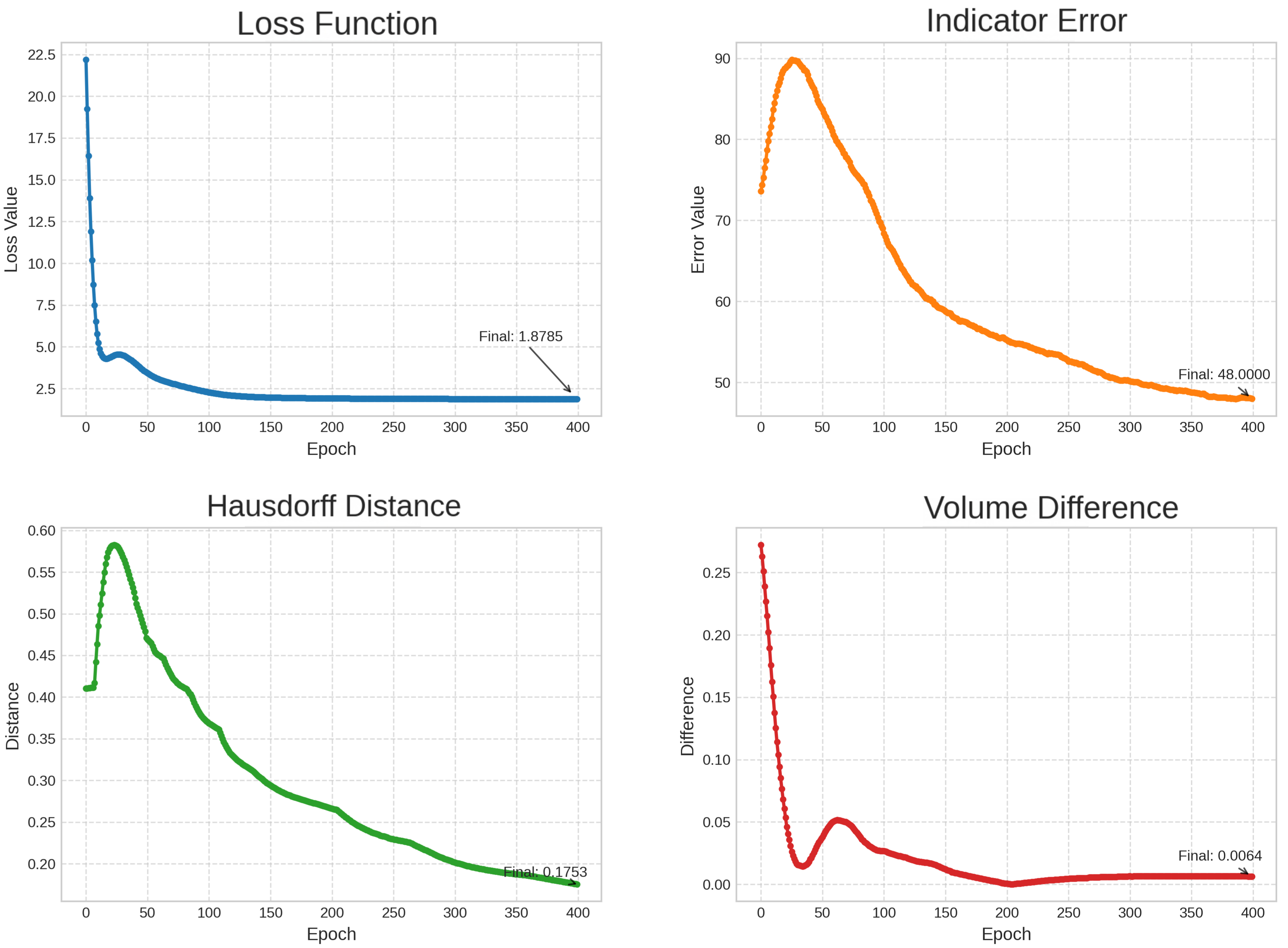}
        \caption{Convergence metrics for the noisy reconstruction (20\% noise level): loss function $\mathcal{L}(z)$ (upper left), indicator error (upper right), Hausdorff difference (lower left), and volume distance (lower right).}
        \label{noise:error}
    \end{figure}
    
    Figure~\ref{example12:result}-\ref{example12:intermidiate} present our second test case for Scenario I, which focuses on the reconstruction of a complex pancreas model with its distinctive anatomical morphology. The multi-view comparison clearly demonstrates our method's ability to accurately capture the intricate geometry of the pancreatic structure, including challenging features such as the pancreatic head, body, and tail regions with their varying thickness and curvature. The convergence behavior of this reconstruction is shown in Figure~\ref{example12:error}, where we can observe the progressive reduction in the four quantitative error metrics. The loss function exhibits rapid initial descent followed by gradual refinement, while the indicator error shows a steady decrease throughout the process. Particularly noteworthy is the Hausdorff distance curve, which reveals how our method systematically reduces geometric discrepancies, achieving a final value of 0.1834. The volume difference metric demonstrates the most dramatic improvement, with early fluctuations eventually stabilizing to reach an impressive final value of 0.0093—confirming the high volumetric accuracy of our reconstruction, which is critical for pancreatic applications.
    \begin{figure}[htbp]
        \centering
        \includegraphics[width=.8\textwidth]{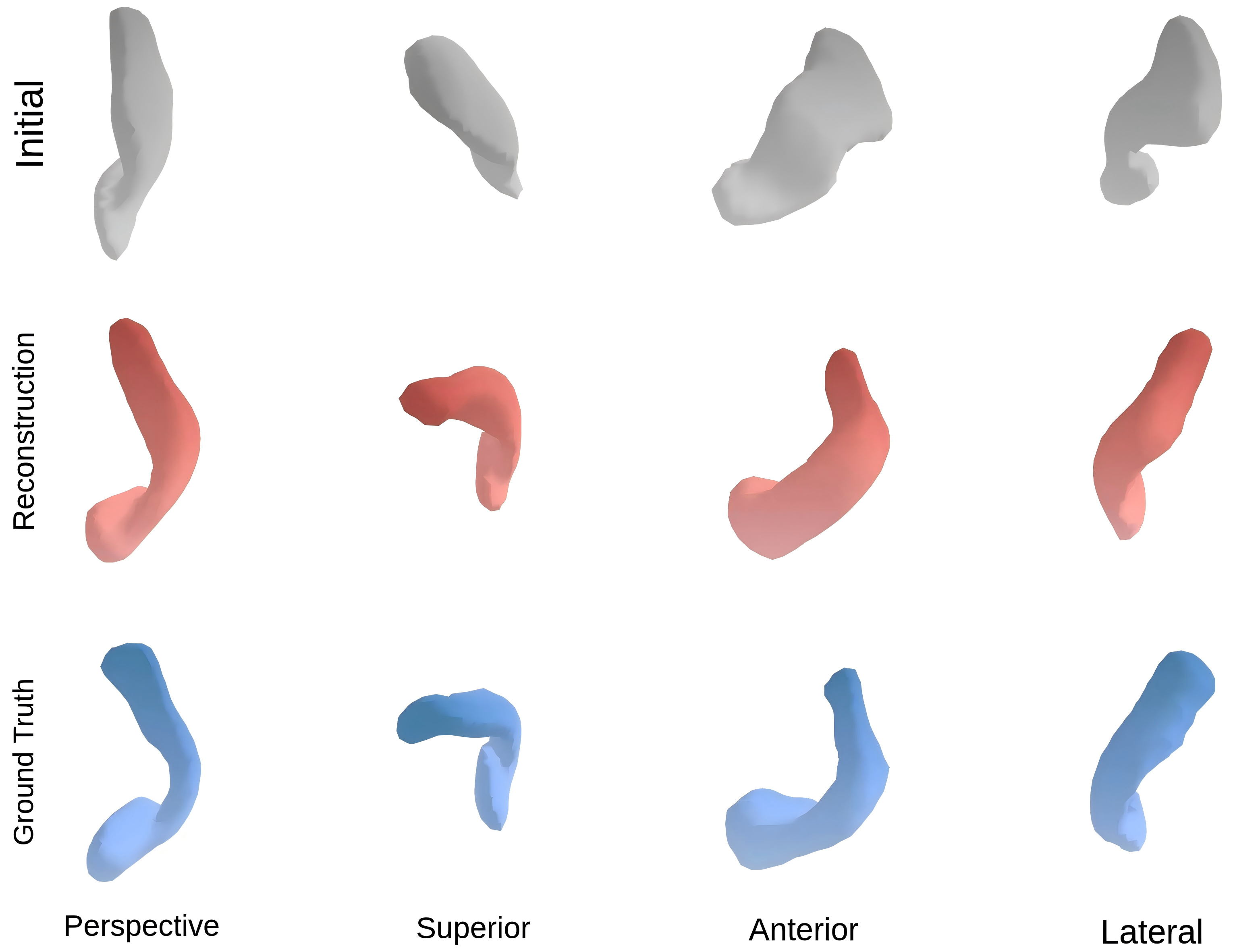}
        \caption{The reconstructions for the second example in Scenario I with exact data: the first, middle, and third rows refer to the initial, the optimal recovery, and the ground truth target, respectively.}
        \label{example12:result}
    \end{figure}
    \begin{figure}[htbp]
        \centering
        \includegraphics[width=0.6\textwidth]{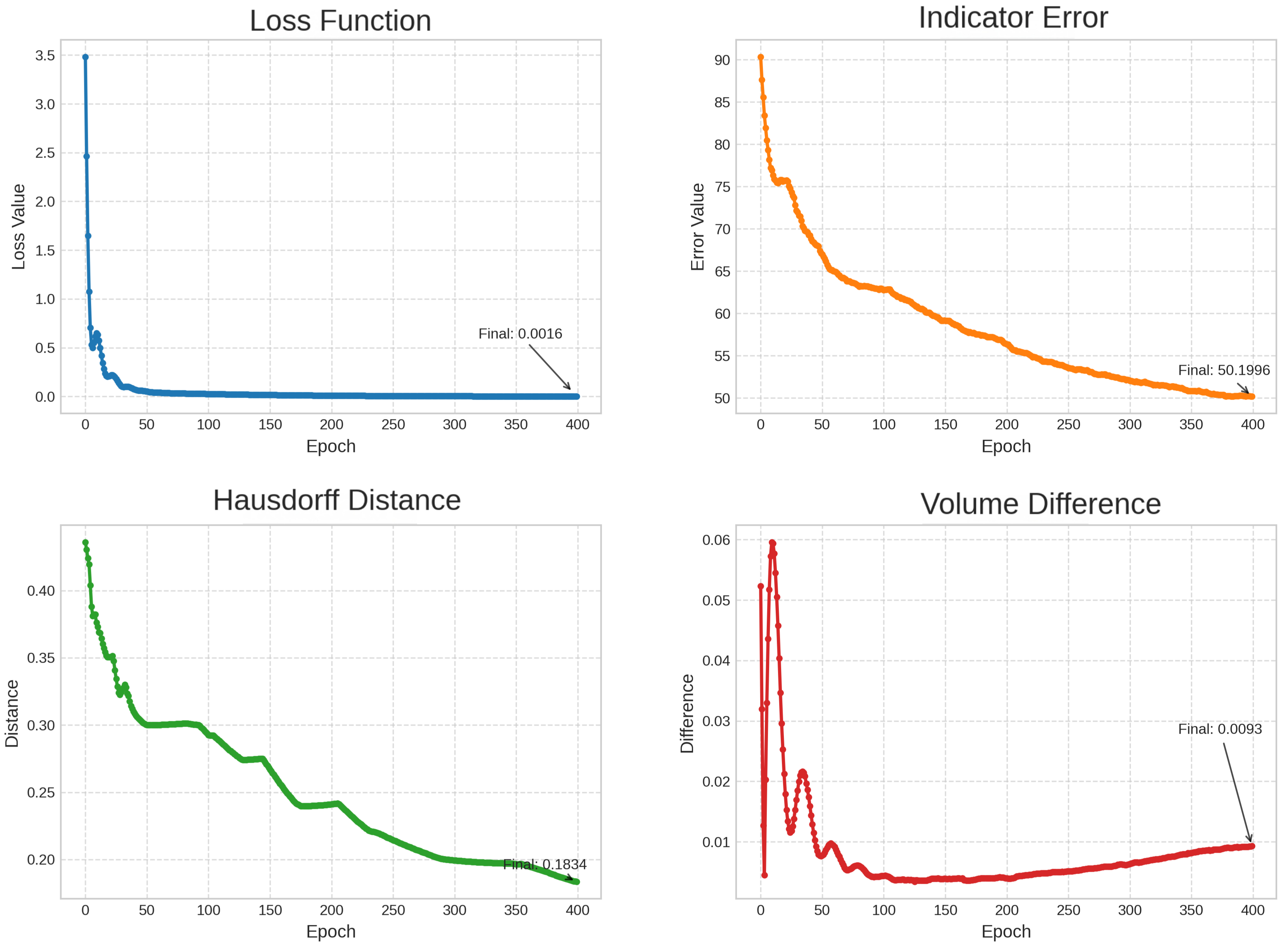}
        \caption{Convergence of the algorithm for the second example in Scenario I in terms of the loss $\mathcal{L}(z)$ (upper left), the indicator error $e$ (upper right), the Hausdorff error $e$ (lower left) and the volume error (lower right).}
        \label{example12:error}
    \end{figure}
    
    Figure~\ref{example12:intermidiate} shows the target and intermediate recovery results at different iterations for our test case, focusing on reconstructing a complex anatomical structure. The visualization presents the progressive shape evolution from a lateral viewpoint, capturing key developmental stages at iterations 0, 30, 60, 90, 120, 150, and 399 (ordered from left to right and from top to bottom), with the final panel showing the exact target shape. The transformation begins with a simple initialization (iteration 0) that roughly approximates the general contour, but lacks anatomical details. As the optimization progresses through iterations 30-60, we observe the emergence of the characteristic curved profile and the initial definition of both the upper and lower regions. By iterations 90-120, the algorithm successfully refines the proportions between different anatomical sections, with noticeable improvements in the representation of the elongated structure and curvature. At iteration 150, most of the key morphological features have been established. The final reconstruction at iteration 399 demonstrates remarkable fidelity to the ground truth model, accurately capturing the distinctive curved shape, varying thickness, and the characteristic bending features that are particularly challenging to reconstruct due to their complex topology. This visual progression clearly illustrates how our method systematically builds anatomical detail through the optimization process.
    \begin{figure}[htbp]
        \centering
        \includegraphics[width=.8\textwidth]{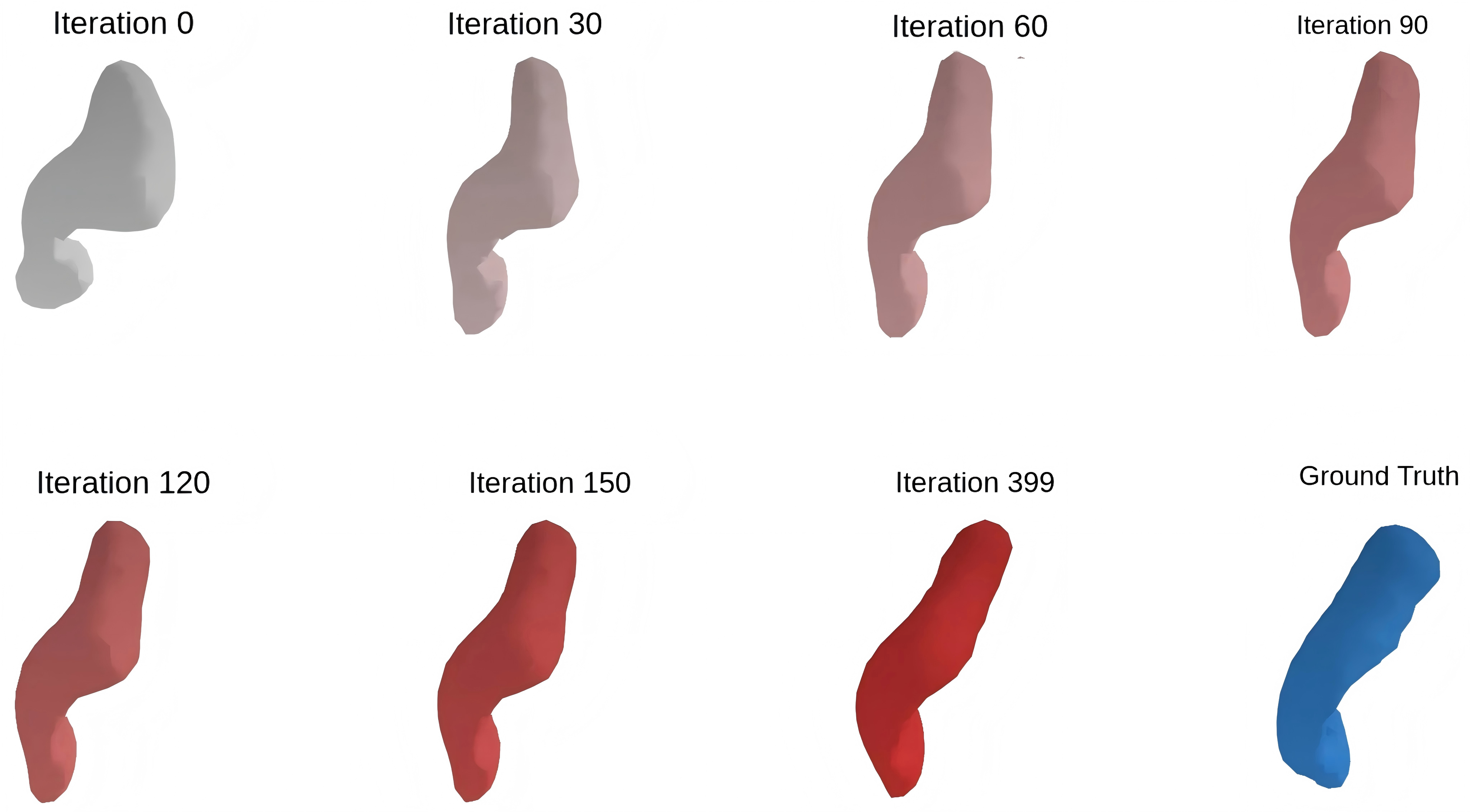}
        \caption{The convergence of the optimization algorithm for the second example in Scenario I: the intermediate reconstruction at iterations: 0, 30, 60, 90, 120, 150, and 399 (ordered from left to right and from top to bottom), and the last plot denotes the ground truth.}
        \label{example12:intermidiate}
    \end{figure}
    
    \subsubsection{Scenario II: Cardiac Reconstruction}
    In the second scenario, we evaluated our method on more challenging cardiac structures that exhibit complex morphological features and intricate surface topology. 
    
    Figure~\ref{fig:sec2_ex0} presents the reconstruction results from the first test case in this scenario, where we demonstrate the progressive evolution of the cardiac model across multiple iterations. The multi-view comparison clearly illustrates how our method transforms from an initial approximation to the final reconstruction, which closely resembles the ground truth model. Particularly impressive is the reconstruction's ability to capture the distinctive ventricular contours, atrial appendages, and the subtle surface variations that are characteristic of cardiac anatomy. When viewed from multiple perspectives—perspective, superior, anterior, and lateral—the reconstruction maintains high fidelity across all viewpoints, confirming the volumetric consistency of our approach. 
    
    Quantitative analysis in Figure~\ref{fig:sec2_loss0} reveals that major structural elements converge within the first 100 iterations, with the loss function decreasing by over 90\% during this initial phase. The Hausdorff distance metric exhibits an interesting temporary fluctuation around iteration 50, which corresponds to the algorithm's transition from global shape optimization to local detail refinement—a critical phase in accurately reconstructing the complex cardiac geometry. In the final iteration, our method achieves impressive error metrics with a Hausdorff distance of 0.2228 and volume difference of 0.0344, representing a 99.9\% improvement from the initial state. These quantitative results, coupled with the visual evidence of accurate reconstruction, validate the effectiveness of our approach for complex anatomical structures where precise geometry is essential for clinical applications such as surgical planning and functional assessment.
    
    \begin{figure}[htbp]
        \centering
        \includegraphics[width=.8\textwidth]{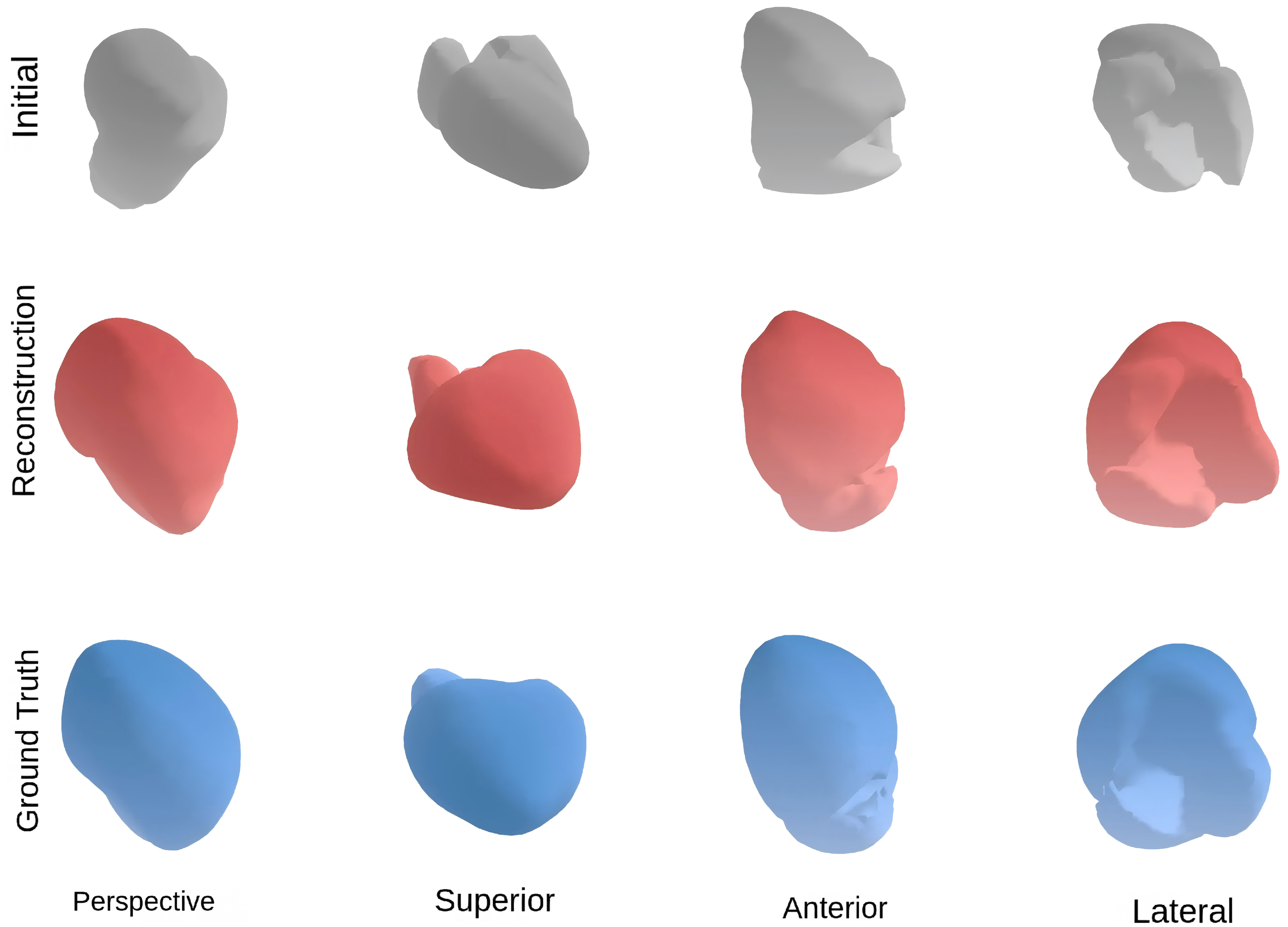}
        \caption{Progressive evolution of 3D shape reconstruction for the first test case in Scenario II.}
        \label{fig:sec2_ex0}
    \end{figure}
    \begin{figure}[htbp]
        \centering
        \includegraphics[width=0.6\textwidth]{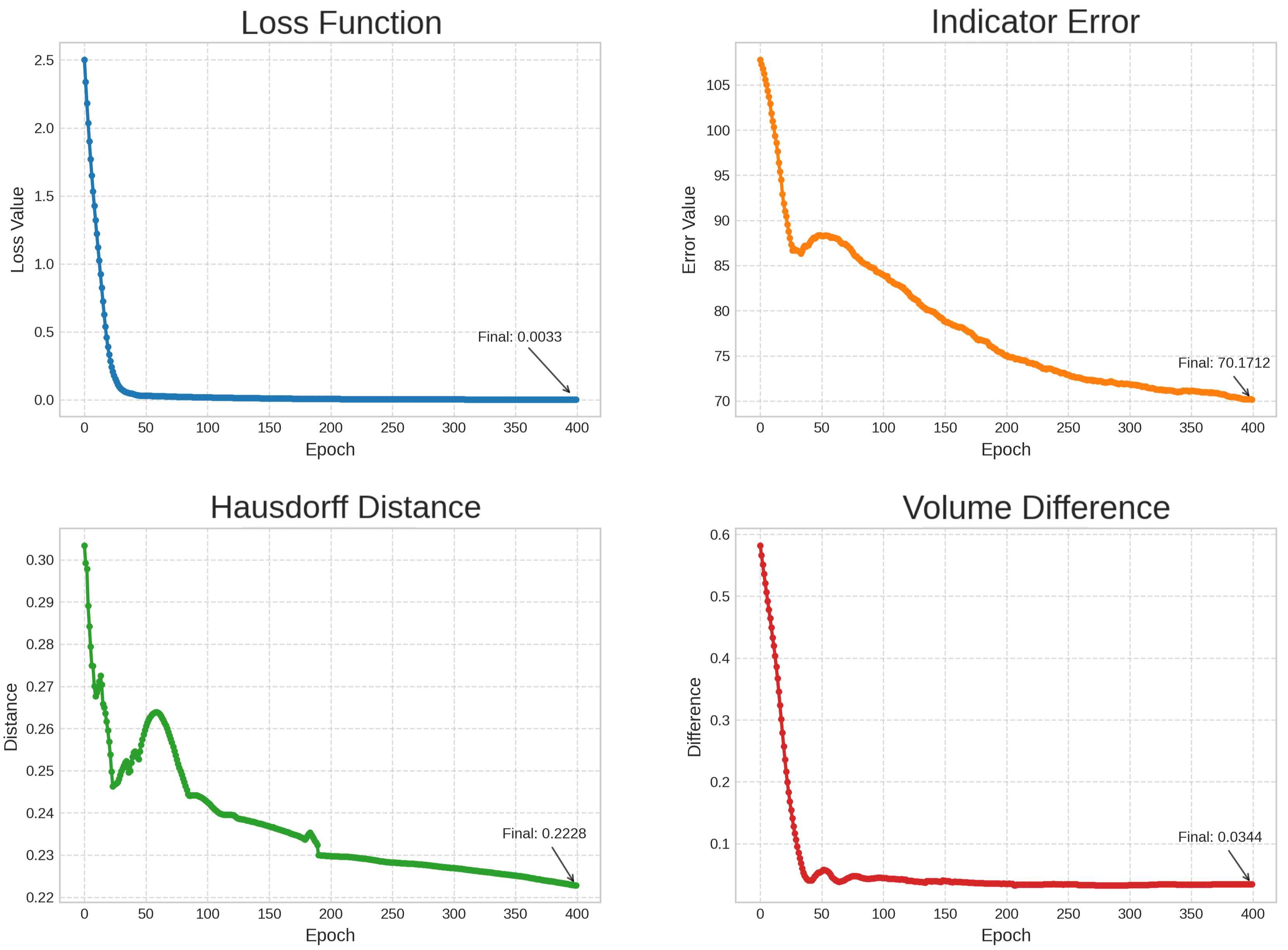}
        \caption{Error curves for the first test case in Scenario II, corresponding to the reconstruction sequence illustrated previously.}
        \label{fig:sec2_loss0}
    \end{figure}
   
    In our second test case of Scenario II, we evaluate the reconstruction performance on a more anatomically complex cardiac model, which presents significant challenges due to its intricate chamber geometry, valvular structures, and detailed myocardial surface features. This particular model incorporates both ventricular and atrial components with varying wall thickness and exhibits complex topological features at the vessel junction areas. Figure~\ref{fig:sec2_ex1} illustrates the multi-view comparison between our reconstruction result and the ground truth model from four complementary perspectives (perspective, superior, anterior, and lateral views). The visual comparison clearly demonstrates our method's ability to accurately recover the complex morphological features, with the reconstructed model (middle row) closely approximating the ground truth (bottom row) across all viewpoints. Notably, our approach successfully captures the distinctive contours of the major vessels and the subtle surface curvatures of the ventricular regions, representing a substantial improvement over the initial estimation (top row). The superior view particularly highlights the accurate reconstruction of the atrial structures, while the lateral view demonstrates the precise recovery of the anatomy of the ventricular outflow tract. Figure~\ref{fig:sec2_loss1} displays the error metrics throughout the optimization process, revealing insightful convergence patterns. The loss function shows exceptional efficiency, rapidly decreasing from 0.42 to below 0.1 within the first 30 iterations and ultimately converging to 0.0006. The indicator error exhibits a distinctive two-phase convergence pattern with an initial rapid descent followed by more gradual refinement, eventually stabilizing at 55.1362. Most interestingly, the Hausdorff distance curve reveals a characteristic temporary increase around iteration 40, corresponding to the transition between global and local optimization phases, before steadily decreasing to a final value of 0.1934. The volume difference metric demonstrates the most dramatic improvement, rapidly reducing to below 0.01 and ultimately achieving an impressive final value of 0.0041 - confirming the high volumetric accuracy of our reconstruction, which is particularly critical for cardiac applications where precise chamber volumes directly impact clinical measurements and functional assessments.
    \begin{figure}[htbp]
        \centering
        \includegraphics[width=.8\textwidth]{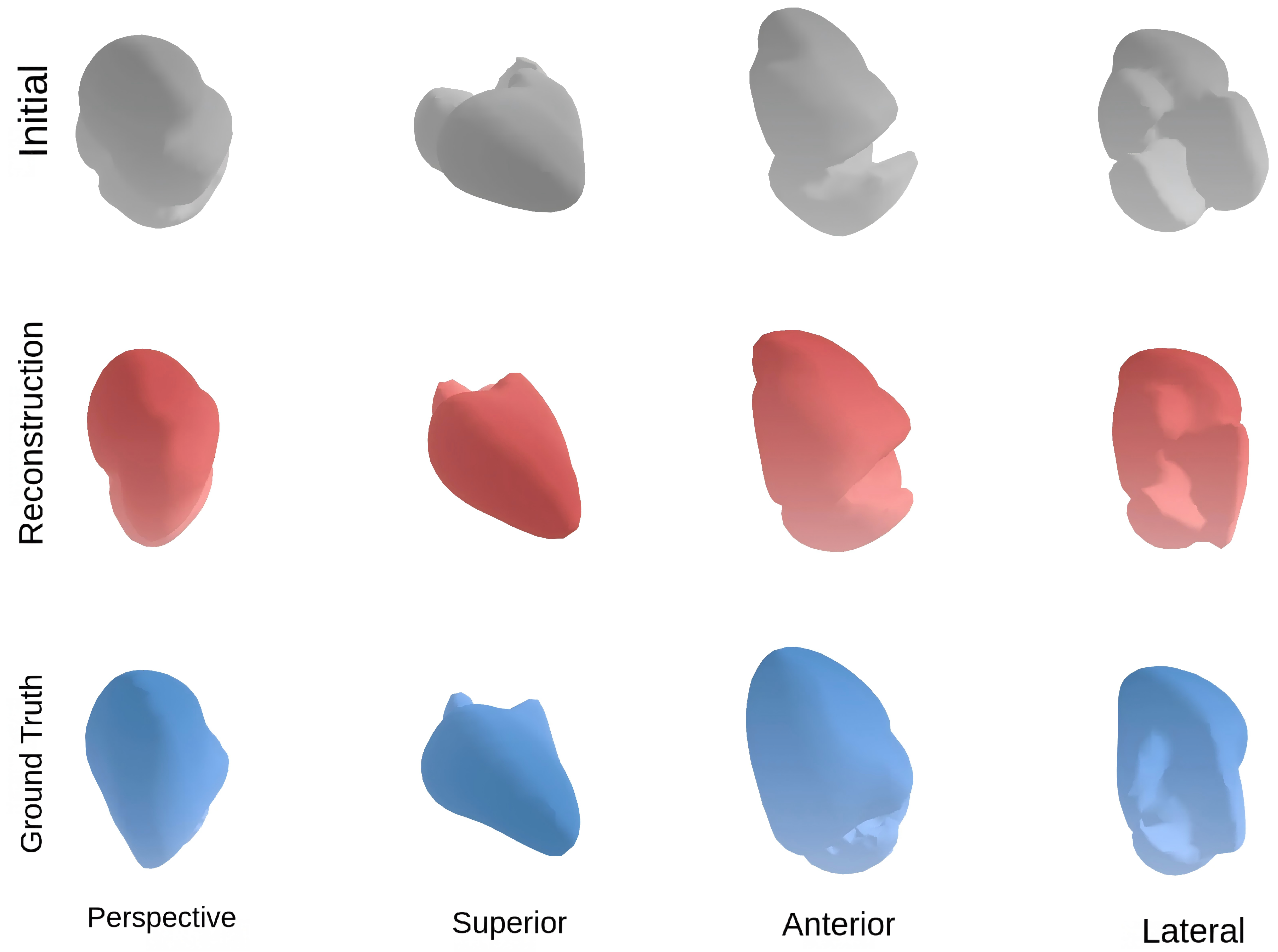}
        \caption{Cardiac model reconstruction in the second test case of Scenario II, shown from multiple viewpoints. The comparison includes the initial model (top row), final reconstruction result (middle row), and ground truth (bottom row) across four different views: perspective, superior, anterior, and lateral. The reconstruction successfully captures the complex morphological features of the cardiac anatomy, including the ventricular regions, atrial structures, and the distinctive contours of the major vessels.}
        \label{fig:sec2_ex1}
    \end{figure}
    \begin{figure}[htbp]
        \centering
        \includegraphics[width=0.6\textwidth]{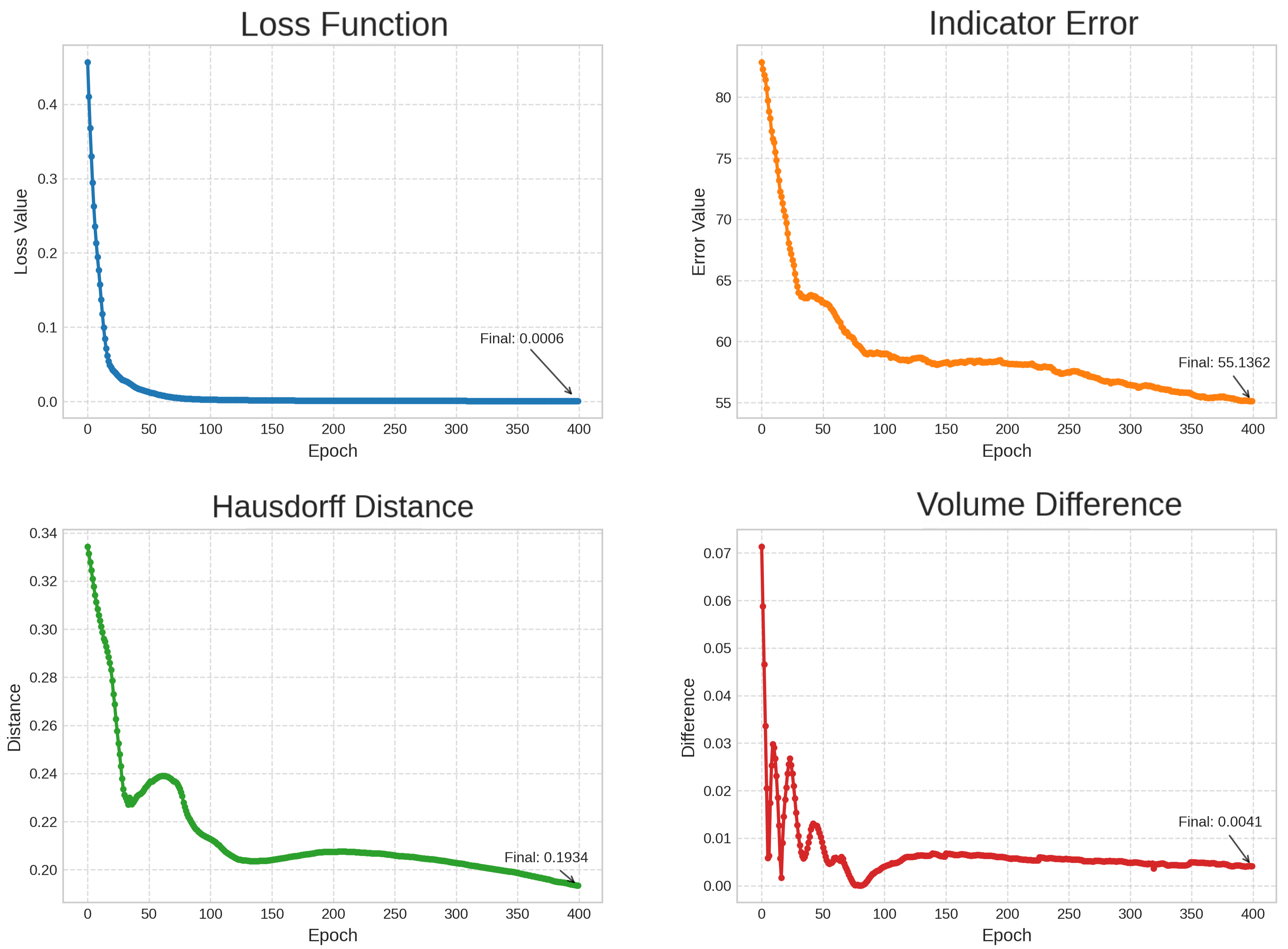}
        \caption{Convergence metrics for the cardiac model reconstruction in the second test case of Scenario II. The plots show the progression of four key error measurements during optimization: Loss Function (top left), Indicator Error (top right), Hausdorff Distance (bottom left), and Volume Difference (bottom right). Each metric demonstrates distinctive convergence patterns, with the loss function and volume difference showing particularly rapid improvement.}
        \label{fig:sec2_loss1}
    \end{figure}
    
    \section{Conclusion}\label{sec:conclusion}
    This work establishes a new paradigm for solving high-dimensional inverse problems by harmonizing the rigor of physics-based modeling with the expressive power of generative AI. We have introduced a ``solver-in-the-loop'' framework for 3D EIT that fundamentally transforms interface reconstruction from an ill-posed search in infinite-dimensional space to a principled optimization on a compact, learned manifold. By enforcing the governing PDE as a hard constraint through a boundary integral solver while leveraging a differentiable neural prior for geometric regularization, our approach overcomes the fragility of traditional shape optimization and the physical inconsistency often plaguing pure deep learning methods. The theoretical convergence guarantees and robust numerical performance on complex anatomical shapes underscore the reliability of this hybrid architecture.
    
    Looking forward, this framework paves the way for a broader class of scientific machine learning algorithms where data-driven priors and physical laws operate in concert rather than in competition. Immediate extensions include tackling anisotropic conductivities and dynamic interface evolution, while the ultimate vision is to generalize this methodology to other imaging modalities where data is scarce but physical principles are well-understood. Such advancements promise to unlock new capabilities in medical diagnostics, non-destructive testing, and beyond, marking a significant step toward trustworthy, physics-aware AI.
    
    \section*{Acknowledgment}
    The work of J. Chen was partially supported by the National Key R\&D Program of China under grants 2019YFA0709600 and 2019YFA0709602. G. Lin would like to acknowledge the support by the National Science Foundation (NSF) under grants DMS-2533878, DMS-2053746, DMS-2134209, ECCS-2328241, CBET-2347401 and OAC-2311848, and by the U.S.~Department of Energy (DOE) Office of Science Advanced Scientific Computing Research program under award number DE-SC0023161, and the DOE–Fusion Energy Science program, under grant number: DE-SC0024583.
    
\FloatBarrier

\bibliographystyle{abbrv}
\bibliography{ref}
\end{document}